\documentclass[12pt, reqno]{amsart}

\usepackage[T1]{fontenc}

\usepackage{amsmath, amssymb, amsfonts, amsthm}
\usepackage{mathtools}
\usepackage{latexsym}

\usepackage{amsbsy}

\usepackage[left=30mm, right=30mm, top=20mm, bottom=20mm]{geometry}

\usepackage{enumitem}
\usepackage{comment}

\usepackage{xcolor}

\usepackage{tikz}
\usetikzlibrary{circuits, intersections}
\usetikzlibrary{plotmarks}
\usetikzlibrary{angles, quotes}
\tikzset{
	node/.style={draw, fill, circle, inner sep=0pt, minimum size=2mm}
}

\usepackage{tikz-cd}
\usetikzlibrary{decorations.pathmorphing}
\usetikzlibrary{circuits, intersections}
\usetikzlibrary{plotmarks}
\usetikzlibrary{angles, quotes}

\usepackage{graphicx}
\usepackage{pdfpages}

\usepackage[super]{nth}

\usepackage{hyperref}
\usepackage[colorinlistoftodos]{todonotes}

\newtheorem{thmintro}{Theorem}

\theoremstyle{definition}
\newtheorem{rkintro}{Remark}

\theoremstyle{plain}
\newtheorem{theorem}{Theorem}[section]
\newtheorem{proposition}[theorem]{Proposition}
\newtheorem{lemma}[theorem]{Lemma}
\newtheorem{corollary}[theorem]{Corollary}

\theoremstyle{definition}

\newtheorem{definition}[theorem]{Definition}
\newtheorem{convention}[theorem]{Convention}
\newtheorem{example}[theorem]{Example}

\theoremstyle{remark}
\newtheorem{remark}[theorem]{Remark}

\renewcommand{\phi}{\varphi}
\renewcommand{\epsilon}{\varepsilon}

\newcommand{\NN}{\mathbb{N}}
\newcommand{\ZZ}{\mathbb{Z}}

\newcommand{\FF}{\mathbb{F}}

\newcommand{\KK}{\mathbb{K}}

\newcommand{\Cp}{\mathcal{C}_{\epsilon}}
\newcommand{\Cm}{\mathcal{C}_{-\epsilon}}
\newcommand{\C}{\mathcal{C}}
\renewcommand{\P}{\mathcal{P}}

\DeclareMathOperator{\Aut}{Aut}

\DeclareMathOperator{\proj}{proj}

\newcommand{\gen}[1]{\langle #1 \rangle}

\usepackage{xcolor}

\numberwithin{equation}{section} 
\begin{document}

\title{RGD-systems of type $(4, 4, 4)$ over $\FF_2$ and tree products}

\author{Sebastian Bischof}

\thanks{email: sebastian.bischof@math.uni-paderborn.de}

\thanks{Mathematisches Institut, Arndtstra\ss e 2, 35392 Gie\ss en, Germany}



\begin{abstract}
	In this paper we prove that the group $U_+$ of an RGD-system of type $(4, 4, 4)$ over $\FF_2$ contains a certain tree product as a subgroup. The proof relies on a careful analysis of the action on the associated twin building. This result is part of a larger project and we will use it as an induction start to construct uncountably many RGD-systems of type $(4, 4, 4)$ over $\FF_2$.
\end{abstract}

\maketitle

\section{Introduction}

In \cite{Ti92} RGD-systems were introduced by Tits in order to study groups of Kac--Moody type. Any RGD-system has a \emph{type} which is given by a Coxeter system, and to each Coxeter system one can associate a set of \emph{roots} $\Phi$ (viewed as half-spaces). Let $(W, S)$ be a Coxeter system, let $\ell$ denote its corresponding length function, and let $\Phi$ be its associated set of roots. An \emph{RGD-system of type $(W, S)$} is a pair $\mathcal{D} = (G, (U_{\alpha})_{\alpha \in \Phi})$, where $G$ is a group and $(U_{\alpha})_{\alpha \in \Phi}$ is a family of subgroups (called \emph{root groups}) indexed by the set of roots $\Phi$, satisfying some axioms. An RGD-system is said to be \emph{over $\FF_2$} if all root groups have cardinality $2$.

One axiom of RGD-systems makes an assumption on commutation relations between root groups corresponding to \emph{prenilpotent pairs of roots}, i.e.\ pairs of roots $\{ \alpha, \beta \} \subseteq \Phi$ satisfying $\alpha \cap \beta \neq \emptyset \neq (-\alpha) \cap (-\beta)$ where $-\alpha$ denotes the root opposite to $\alpha$. It is a basic fact that for any prenilpotent pair $\{ \alpha, \beta \}$ the group $\langle U_{\alpha}, U_{\beta} \rangle$ is nilpotent provided that $U_{\alpha}$ and $U_{\beta}$ are nilpotent (which is the case for RGD-systems over $\FF_2$). More general, if $\Psi \subseteq \Phi$ is a \emph{nilpotent set of roots}, i.e.\ $\bigcap_{\alpha \in \Psi} \alpha \neq \emptyset \neq \bigcap_{\alpha \in \Psi} (-\alpha)$, then the subgroup $\langle U_{\alpha} \mid \alpha \in \Psi \rangle$ is nilpotent provided that the root groups are nilpotent. Important examples of such subgroups of $G$ are the groups
\[ U_w := \langle U_{\alpha} \mid \alpha \in \Phi,1_W \in \alpha, w\notin \alpha \rangle \]
for $w\in W$. Let $\Phi_+ := \{ \alpha \in \Phi \mid 1_W \in \alpha \}$ be the set of \emph{positive} roots. Then the structure of $U_+ := \langle U_{\alpha} \mid \alpha \in \Phi_+ \rangle$ is known: According to a result of Tits \cite{Ti86} it is the direct limit of the inductive system formed by the groups $(U_w)_{w\in W}$ together with the natural inclusions $U_w \to U_{ws}$ if $\ell(ws) = \ell(w) +1$. Thus a presentation of the group $U_+$ is given by the generating set $\{ U_{\alpha} \mid \alpha \in \Phi_+ \}$ and all the relations are given by the relations inside the root groups $U_{\alpha}$ and by the commutation relations between root groups corresponding to prenilpotent pairs of roots. In particular, we do not see a relation between root groups corresponding to a pair of roots $\{ \alpha, \beta \} \subseteq \Phi_+$ which is not prenilpotent. In fact, for such a pair the group generated by $U_{\alpha}$ and $U_{\beta}$ is isomorphic to their free product, i.e.\ $\langle U_{\alpha}, U_{\beta} \rangle \cong U_{\alpha} \star U_{\beta}$. 

Suppose that $(W, S)$ is \emph{of type $(4, 4, 4)$}, that is, $\vert S \vert = 3$ and $o(st) = 4$ for all $s \neq t \in S$, and let $(G, (U_{\alpha})_{\alpha \in \Phi})$ be an RGD-system of type $(W, S)$ over $\FF_2$ (e.g.\ a split Kac--Moody group of type $(4, 4, 4)$ over $\FF_2$). We have already mentioned that we know the structure of the group $U_+ = \langle U_{\alpha} \mid \alpha \in \Phi_+ \rangle$. In this paper we will determine the structure of the group $\langle U_{\alpha} \mid \alpha \in \Psi \rangle$ for the non-nilpotent set $\Psi = \{ \alpha \in \Phi_+ \mid \{ sr, trt \} \not\subseteq \alpha \}$, where $S = \{ r, s, t \}$ (for more information about why we consider this specific $\Psi$ we refer to Remark \ref{rkintro: Choice of Psi}).

For distinct $s, t \in S$ we denote the unique element in $\gen{s, t}$ of maximal length by $r_{\{s, t\}}$ ($= stst = tsts$). If $w\in W$ is such that $\ell(ws) = \ell(wt) = \ell(w) +1$, then we define
\[ V_{w r_{\{s, t\}}} := \langle U_{ws}, U_{wt} \rangle \leq U_{w r_{\{s, t\}}}. \]
We denote by $U_{sr} \star_{U_s} V_{r_{\{s, t\}}} \star_{U_t} U_{trt}$ the \emph{tree product} of the following tree of groups: the tree is a sequence of length $2$, the vertex groups are $U_{sr}, V_{r_{\{s, t\}}}, U_{trt}$ and the edge groups are $U_s, U_t$. By definition we have a canonical homomorphism from $U_{sr} \star_{U_s} V_{r_{\{s, t\}}} \star_{U_t} U_{trt}$ to $G$. We prove the following main result (cf.\ Theorem \ref{theoremsubgroupKM}):

\begin{thmintro}\label{Main theorem: V_Rs to G injective}
	Let $(W, S)$ be of type $(4, 4, 4)$ and let $(G, (U_{\alpha})_{\alpha \in \Phi})$ be an RGD-system of type $(W, S)$ over $\FF_2$. Then the canonical homomorphism $U_{sr} \star_{U_s} V_{r_{\{s, t\}}} \star_{U_t} U_{trt} \to G$ is injective.
\end{thmintro}

\begin{rkintro}
	Theorem \ref{Main theorem: V_Rs to G injective} becomes false if we replace $U_{sr}$ by $U_{srs}$, because $U_{r\alpha_s} \leq U_{srs}$, $U_{r\alpha_t} \leq U_{trt}$ and $\{ r\alpha_s, r\alpha_t \}$ is a prenilpotent pair of roots.
\end{rkintro}

\begin{rkintro}
	Theorem \ref{Main theorem: V_Rs to G injective} is only true over $\FF_2$. If the root groups contain more than two elements, then it follows from \cite[Lemma $18$ and Proposition $7$]{Ab96} that $ U_{s\alpha_t} \leq V_{r_{\{s, t\}}}$. But for $1 \neq u \in U_{s\alpha_r}$ and $1 \neq v \in U_{s\alpha_t}$ we have $[u, [u, [u, v]]] = 1$ in $G$. As $[u, [u, [u, v]]] \neq 1$ in $U_{sr} \star_{U_s} V_{r_{\{s, t\}}}$, we conclude that the mapping $U_{sr} \star_{U_s} V_{r_{\{s, t\}}} \to G$ has a non-trivial kernel.
\end{rkintro}

\subsection*{Application}

Commutator blueprints were introduced in \cite{BiRGD} as purely combinatorial objects which prescribe the commutation relations between prenilpotent pairs of \emph{positive} roots. In particular, they determine the groups $U_w$ mentioned above. The main result of \cite{BiRGD} implies that if one wants to construct RGD-systems over $\FF_2$, it suffices to construct commutator blueprints which are \emph{Weyl-invariant} (a stronger version of locally Weyl-invariance; cf.\ Definition \ref{Definition: commutator blueprint}) and \emph{faithful} (i.e.\ the canonical mappings $U_w \to \lim U_w$ are injective). In \cite{BiConstruction} we constructed Weyl-invariant commutator blueprints of type $(4, 4, 4)$. The remaining question whether (locally) Weyl-invariant commutator blueprints of type $(4, 4, 4)$ are also faithful will be answered in \cite{BiUncountablymanyRGD-systems} where we will use Theorem \ref{Theorem: Main application} and Corollary \ref{Corollary: Main application} for the base case of an induction. Corollary~\ref{Corollary: Main application} is a consequence of Theorem~\ref{Theorem: Main application} which is itself a consequence of Theorem~\ref{Main theorem: V_Rs to G injective}. Let us briefly explain the induction strategy.

To see that a (locally) Weyl-invariant commutator blueprint of type $(4, 4, 4)$ is faithful, we need to show that the canonical mappings $U_w \to \lim U_w$ are injective. This will be done by constructing a family of groups $(G_i)_{i\in \NN}$ with canonical homomorphisms $G_i \to G_{i+1}$ such that the direct limit $\lim G_i$ is isomorphic to the direct limit $\lim U_w$ (for the definition of the group $G_0$ we refer to Definition~\ref{Definition: Gi}). Since $\lim G_i$ is a sequence of groups, the mappings $G_i \to \lim G_i$ are injective if the mappings $G_i \to G_{i+1}$ are injective for each $i \in \NN$. This will be done in \cite{BiUncountablymanyRGD-systems} by induction on $i$. In the present paper we prove a crucial step for the base case of the induction, that is, the mapping $G_0 \to G_1$ is injective. To prove this, we will mainly use the following theorem (cf.\ Theorem~\ref{Theorem: Main application}).

\begin{thmintro}\label{Main Theorem: Application}
	There exists a canonical homomorphism 
	\[ U_{s r_{\{r, t\}}} \star_{U_{str}} V_{st r_{\{r, s\}}} \star_{U_{sts}} U_{r_{\{s, t\}}} \star_{U_{tst}} V_{ts r_{\{r, t\}}} \star_{U_{tsr}} U_{t r_{\{r, s\}}} \to G_0. \]
	Moreover, this homomorphism is injective.
\end{thmintro}

\begin{rkintro}\label{rkintro: Choice of Psi}
	We could have tried to prove a stronger version of Theorem~\ref{Main theorem: V_Rs to G injective}, that is, to prove that the tree product in Theorem~\ref{Main Theorem: Application} embeds into $G$. This corresponds to the set $\Psi = \{ \alpha \in \Phi_+ \mid \{ srtrt, stst, trsrs \} \not\subseteq \alpha \}$. But then the computations would have become much more technical. However, using the main result of \cite{BiUncountablymanyRGD-systems}, it will follow that $G_0 \to G$ is injective showing that, indeed, the tree product mentioned in Theorem~\ref{Main Theorem: Application} embeds into any Kac--Moody group of type $(4, 4, 4)$ over $\FF_2$.
\end{rkintro}

\subsection*{Overview}

In Section \ref{Section: Preliminaries} we recall the basic definitions and fix notation. Moreover, we prove some elementary results which we will use later. In Section \ref{Section: Action on building} we first prove the technical Lemmas \ref{Lemma: Uplus acts on Deltaminus} and \ref{Lemma: Normal form} which we then use to prove Theorem \ref{theoremsubgroupKM} (this is Theorem~\ref{Main theorem: V_Rs to G injective}). The goal of Section \ref{Section: Application} is to prove Theorem \ref{Theorem: Main application}. The proof of this theorem uses a lot of Bass--Serre theory. We will first introduce several tree products and show that some of them are subgroups of others. Moreover, we will show that certain tree products are isomorphic to other tree products. For all the computations in tree products we will heavily use Proposition \ref{treeproducts}, Proposition \ref{treeofgroupsinjective} and Lemma \ref{folding}.

\subsection*{Acknowledgement}

I would like to thank Bernhard M\"uhlherr, Richard Weidmann and Stefan Witzel for helpful discussions. I would also like to thank the anonymous referee for useful comments.

\section{Preliminaries}\label{Section: Preliminaries}

\subsection*{Coxeter systems}

Let $(W, S)$ be a Coxeter system and let $\ell$ denote the corresponding length function. The \emph{rank} of the Coxeter system is the cardinality of the set $S$. For the purpose of this paper, we assume that all Coxeter systems are of finite rank.

\begin{convention}
	In this paper we let $(W, S)$ be a Coxeter system of finite rank.
\end{convention}

It is well-known that for each $J \subseteq S$ the pair $(\langle J \rangle, J)$ is a Coxeter system (cf.\ \cite[Ch. IV, §$1$ Theorem $2$]{Bo02}). A subset $J \subseteq S$ is called \textit{spherical} if $\langle J \rangle$ is finite. Given a spherical subset $J$ of $S$, there exists a unique element of maximal length in $\langle J \rangle$, which we denote by $r_J$ (cf.\ \cite[Corollary $2.19$]{AB08}). For $s, t \in S$ we denote the order of $st$ in $W$ by $m_{st}$. The Coxeter system $(W, S)$ is said to be \emph{of type $(4, 4, 4)$} if it is of rank $3$ and $m_{st} = 4$ for all $s\neq t\in S$.

\subsection*{Buildings}

A \textit{building of type $(W, S)$} is a pair $\Delta = (\C, \delta)$ where $\C$ is a non-empty set and where $\delta: \C \times \C \to W$ is a \textit{distance function} satisfying the following axioms, where $x, y\in \C$ and $w = \delta(x, y)$:
\begin{enumerate}[label=(Bu\arabic*)]
	\item $w = 1_W$ if and only if $x=y$;
	
	\item if $z\in \C$ satisfies $s := \delta(y, z) \in S$, then $\delta(x, z) \in \{w, ws\}$, and if, furthermore, $\ell(ws) = \ell(w) +1$, then $\delta(x, z) = ws$;
	
	\item if $s\in S$, there exists $z\in \C$ such that $\delta(y, z) = s$ and $\delta(x, z) = ws$.
\end{enumerate}
The \textit{rank} of $\Delta$ is the rank of the underlying Coxeter system. The elements of $\C$ are called \textit{chambers}. Given $s\in S$ and $x, y \in \C$, then $x$ is called \textit{$s$-adjacent} to $y$, if $\delta(x, y) = s$. The chambers $x, y$ are called \textit{adjacent}, if they are $s$-adjacent for some $s\in S$. A \textit{gallery} from $x$ to $y$ is a sequence $(x = x_0, \ldots, x_k = y)$ such that $x_{l-1}$ and $x_l$ are adjacent for all $1 \leq l \leq k$; the number $k$ is called the \textit{length} of the gallery. Let $(x_0, \ldots, x_k)$ be a gallery and suppose $s_i \in S$ with $\delta(x_{i-1}, x_i) = s_i$. Then $(s_1, \ldots, s_k)$ is called the \textit{type} of the gallery. A gallery from $x$ to $y$ of length $k$ is called \textit{minimal} if there is no gallery from $x$ to $y$ of length $<k$. In this case we have $\ell(\delta(x, y)) = k$ (cf.\ \cite[Corollary $5.17(1)$]{AB08}). Let $x, y, z \in \C$ be chambers such that $\ell(\delta(x, y)) = \ell(\delta(x, z)) + \ell(\delta(z, y))$. Then the concatenation of a minimal gallery from $x$ to $z$ and a minimal gallery from $z$ to $y$ yields a minimal gallery from $x$ to $y$. For two chambers $c, d \in \C$ we abbreviate $\ell(c, d) := \ell(\delta(c, d))$.

Given a subset $J \subseteq S$ and $x\in \C$, the \textit{$J$-residue of $x$} is the set $R_J(x) := \{y \in \C \mid \delta(x, y) \in \langle J \rangle \}$. Each $J$-residue is a building of type $(\langle J \rangle, J)$ with the distance function induced by $\delta$ (cf.\ \cite[Corollary $5.30$]{AB08}). A \textit{residue} is a subset $R$ of $\C$ such that there exist $J \subseteq S$ and $x\in \C$ with $R = R_J(x)$. Since the subset $J$ is uniquely determined by $R$, the set $J$ is called the \textit{type} of $R$ and the \textit{rank} of $R$ is defined to be the cardinality of $J$. A residue is called \textit{spherical} if its type is a spherical subset of $S$. A \textit{panel} is a residue of rank $1$. An \textit{$s$-panel} is a panel of type $\{s\}$ for $s\in S$. For $c\in \C$ and $s\in S$ we denote the $s$-panel containing $c$ by $\P_s(c)$. The building $\Delta$ is called \textit{thick}, if each panel of $\Delta$ contains at least three chambers.

Given $x\in \C$ and a $J$-residue $R \subseteq \C$, then there exists a unique chamber $z\in R$ such that $\ell(\delta(x, y)) = \ell(\delta(x, z)) + \ell(\delta(z, y))$ holds for each $y\in R$ (cf.\ \cite[Proposition $5.34$]{AB08}). The chamber $z$ is called the \textit{projection of $x$ onto $R$} and is denoted by $\proj_R x$. Moreover, if $z = \proj_R x$ we have $\delta(x, y) = \delta(x, z) \delta(z, y)$ for each $y\in R$. Let $R \subseteq T$ be two residues of $\Delta$. Then $\proj_R c = \proj_R \proj_T c$ holds for all $c\in \C$ by \cite[Proposition $2$]{DS87}.

An \textit{(type-preserving) automorphism} of a building $\Delta = (\C, \delta)$ is a bijection $\phi:\C \to \C$ such that $\delta(\phi(c), \phi(d)) = \delta(c, d)$ holds for all chambers $c, d \in \C$. We remark that some authors distinguish between automorphisms and type-preserving automorphisms. An automorphism in our sense is type-preserving. We denote the set of all automorphisms of the building $\Delta$ by $\Aut(\Delta)$. It is a basic fact that the projection commutes with each automorphism. More precisely, let $c \in \C$, let $R$ be a residue of $\Delta$ and let $\phi \in \Aut(\Delta)$. It follows directly from the uniqueness of $\proj_R c$ that $\phi(\proj_R c) = \proj_{\phi(R)} \phi(c)$.

\begin{example}
	We define $\delta: W \times W \to W, (x, y) \mapsto x^{-1}y$. Then $\Sigma(W, S) := (W, \delta)$ is a building of type $(W, S)$, which we call the \emph{Coxeter building} of type $(W, S)$. The group $W$ acts faithful on $\Sigma(W, S)$ by multiplication from the left, i.e.\ $W \leq \Aut(\Sigma(W, S))$.
\end{example}

A subset $\Sigma \subseteq \C$ is called \textit{convex} if for any two chambers $c, d \in \Sigma$ and any minimal gallery $(c_0 = c, \ldots, c_k = d)$, we have $c_i \in \Sigma$ for all $0 \leq i \leq k$. A subset $\Sigma \subseteq \C$ is called \textit{thin} if $P \cap \Sigma$ contains exactly two chambers for every panel $P \subseteq \C$ which meets $\Sigma$. An \textit{apartment} is a non-empty subset $\Sigma \subseteq \C$, which is convex and thin.

\subsection*{Roots}

A \textit{reflection} is an element of $W$ that is conjugate to an element of $S$. For $s\in S$ we let $\alpha_s := \{ w\in W \mid \ell(sw) > \ell(w) \}$ be the \textit{simple root} corresponding to $s$. A \textit{root} is a subset $\alpha \subseteq W$ such that $\alpha = v\alpha_s$ for some $v\in W$ and $s\in S$. We denote the set of all roots by $\Phi := \Phi(W, S)$. The set $\Phi_+ := \{ \alpha \in \Phi \mid 1_W \in \alpha \}$ is the set of all \textit{positive roots} and $\Phi_- := \{ \alpha \in \Phi \mid 1_W \notin \alpha \}$ is the set of all \textit{negative roots}. For each root $\alpha \in \Phi$, the complement $-\alpha := W \backslash \alpha$ is again a root; it is called the root \emph{opposite} to $\alpha$. We denote the unique reflection which interchanges these two roots by $r_{\alpha} \in W \leq \Aut(\Sigma(W, S))$. A pair $\{ \alpha, \beta \}$ of roots is called \textit{prenilpotent} if both $\alpha \cap \beta$ and $(-\alpha) \cap (-\beta)$ are non-empty sets. For such a pair we will write $\left[ \alpha, \beta \right] := \{ \gamma \in \Phi \mid \alpha \cap \beta \subseteq \gamma \text{ and } (-\alpha) \cap (-\beta) \subseteq -\gamma \}$ and $(\alpha, \beta) := \left[ \alpha, \beta \right] \backslash \{ \alpha, \beta \}$.

\subsection*{Coxeter buildings}

In this section we consider the Coxeter building $\Sigma(W, S)$. For $\alpha \in \Phi$ we denote by $\partial \alpha$ the set of all panels stabilized by $r_{\alpha}$. The set $\partial \alpha$ is called the \textit{wall} associated with $\alpha$. Let $G = (c_0, \ldots, c_k)$ be a gallery. We say that $G$ \textit{crosses the wall $\partial \alpha$} if there exists $1 \leq i \leq k$ such that $\{ c_{i-1}, c_i \} \in \partial \alpha$. It is a basic fact that a minimal gallery crosses a wall at most once (cf.\ \cite[Lemma $3.69$]{AB08}). Let $(c_0, \ldots, c_k)$ and $(d_0 = c_0, \ldots, d_k = c_k)$ be two minimal galleries from $c_0$ to $c_k$ and let $\alpha \in \Phi$. Then $\partial \alpha$ is crossed by the minimal gallery $(c_0, \ldots, c_k)$ if and only if it is crossed by the minimal gallery $(d_0, \ldots, d_k)$. For a minimal gallery $G = (c_0, \ldots, c_k)$, $k \geq 1,$ we denote the unique root containing $c_{k-1}$ but not $c_k$ by $\alpha_G$. For a minimal gallery $G = (c_0, \ldots, c_k)$ the \emph{sequence of roots crossed by $G$} is the unique sequence of roots $(\alpha_1, \ldots, \alpha_k)$ such that $c_{i-1} \in \alpha_i$ and $c_i \notin \alpha_i$ for all $1 \leq i \leq k$. We denote the set of all minimal galleries $(c_0 = 1_W, \ldots, c_k)$ starting at $1_W$ by $\mathrm{Min}$. For $w\in W$ we denote the set of all $G \in \mathrm{Min}$ of type $(s_1, \ldots, s_k)$ with $w = s_1 \cdots s_k$ by $\mathrm{Min}(w)$. For $w\in W$ and $s\in S$ with $\ell(sw) = \ell(w) -1$ we let $\mathrm{Min}_s(w)$ be the set of all $G \in \mathrm{Min}(w)$ of type $(s, s_2, \ldots, s_k)$. We extend this notion to the case $\ell(sw) = \ell(w) +1$ by defining $\mathrm{Min}_s(w) := \mathrm{Min}(w)$. Let $w\in W$, $s\in S$ and $G = (c_0, \ldots, c_k) \in \mathrm{Min}_s(w)$. If $\ell(sw) = \ell(w) -1$, then $c_1 = s$ and we define $sG := (sc_1 = 1_W, \ldots, sc_k) \in \mathrm{Min}(sw)$. If $\ell(sw) = \ell(w) +1$, we define $sG := (1_W, sc_0 = s, \ldots, sc_k) \in \mathrm{Min}(sw)$.

\begin{convention}
	For the rest of this subsection we assume that $(W, S)$ is of type $(4, 4, 4)$ and that $S = \{r, s, t\}$.
\end{convention}

\begin{lemma}[{\cite[Lemma $2.14$]{BiConstruction}}]\label{wordsincoxetergroup}
	Suppose $w\in W$ with $\ell(ws) = \ell(w) +1 = \ell(wt)$ and suppose $w' \in \langle s, t\rangle$ with $\ell(w') \geq 2$. Then $\ell(ww'rf) = \ell(w) + \ell(w') +1 + \ell(f)$ for each $f \in \{1_W, s, t\}$.
\end{lemma}

\begin{lemma}[see {\cite[Lemma $3.4$]{Bischof_On_Growth_Functions_of_Coxeter_Groups}} and {\cite[Lemma $2.16$]{BiConstruction}}]\label{Lemma: not both down}
	Suppose $w \in W$ with $\ell(ws) = \ell(w) +1 = \ell(wt)$. Then $\ell(w) +2 \in \{ \ell(wsr), \ell(wtr) \}$. Moreover, if $\ell(wsr) = \ell(w)$, then $\ell(wsrt) = \ell(w)+1$.
\end{lemma}

\begin{lemma}[{\cite[Lemma $2.18$]{BiConstruction}}]\label{mingallinrep}
	Let $H = (d_0, \ldots, d_4)$ be a minimal gallery of type $(r, s, t, r)$ and let $\beta \in \Phi$ with $\{ d_0, d_1 \} \in \partial \beta$ and $d_0 \in \beta$. Then $\beta \subsetneq \gamma$ for each $\gamma \in \{ \alpha_{(d_0, \ldots, d_3)}, \alpha_{(d_0, \ldots, d_4)} \}$.
\end{lemma}

\begin{lemma}\label{Lemma: Subset contained in certain root}
	We have $tstr\alpha_s \cap stsr\alpha_t \cap (W\backslash \{ r_{\{s, t\}}r \}) \subseteq r_{\{s, t\}} \alpha_r$.
\end{lemma}
\begin{proof}
	Let $w \in tstr\alpha_s \cap stsr \alpha_t$ be an element. We show that $w = r_{\{s, t\}} r$ or that $r_{\{s, t\}} w \in \alpha_r$, i.e.\ $\ell(r r_{\{s, t\}} w) = \ell(r_{\{s, t\}} w) +1$. We distinguish the following cases:
	\begin{enumerate}[label=(\roman*)]
		\item $\ell(w^{-1}) +2 \in \{ \ell(w^{-1}ts), \ell(w^{-1}st) \}$: Then $\ell(w^{-1} r_{\{s, t\}}) \geq \ell(\proj_{R_{\{s, t\}(w^{-1})}} 1_W) +2$ and we deduce $\ell(rr_{\{s, t\}}w) = \ell( w^{-1} r_{\{s, t\}}r ) = \ell(w^{-1} r_{\{s, t\}}) +1 = \ell(r_{\{s, t\}}w) +1$ from Lemma \ref{wordsincoxetergroup}.
		
		\item $\ell(w^{-1}s) = \ell(w^{-1}) +1$ and $\ell(w^{-1}st) = \ell(w^{-1})$: By assumption, we have $w\in tstr\alpha_s$ and hence $\ell\left(s(rtstw)\right) = \ell(rtstw) +1$. This implies $\ell(w^{-1}tstrs) = \ell(w^{-1}tstr) +1$ and, in particular, $\ell(w^{-1}r_{\{s, t\}}r) = \ell(w^{-1} r_{\{s, t\}}) +1$.
		
		\item $\ell(w^{-1}t) = \ell(w^{-1}) +1$ and $\ell(w^{-1}ts) = \ell(w^{-1})$: This follows similar as in the previous case.
		
		\item $\ell(w^{-1}s) = \ell(w^{-1}) -1 = \ell(w^{-1}t)$: If $\ell(rr_{\{s, t\}}w) = \ell(r_{\{s, t\}}w) +1$ there is nothing to show. Thus we suppose $\ell(rr_{\{s, t\}}w) = \ell(r_{\{s, t\}}w) -1$. Assume that $\ell(w^{-1}stsr) = \ell(w^{-1} sts)-1$. Then we would have $\ell(w^{-1}stsrt) = \ell(w^{-1}sts)-2$, which is a contradiction to the assumption $w \in stsr\alpha_t$. Thus we have $\ell(w^{-1}stsr) = \ell(w^{-1}sts) +1$. Using the fact that $w \in tstr\alpha_s \cap stsr\alpha_t$, we infer $\ell(w^{-1}tstrs) = \ell(w^{-1}tst) +2$ and $\ell(w^{-1}stsrt) = \ell(w^{-1}sts)+2$. This yields $\ell(w^{-1}r_{\{s, t\}}ru) = \ell(w^{-1}r_{\{s, t\}}r) +1$ for each $u\in S = \{r, s, t\}$ and hence $w^{-1}r_{\{s, t\}}r = 1$. \qedhere
	\end{enumerate}
\end{proof}

\subsection*{Twin buildings}

Let $\Delta_+ = (\C_+, \delta_+), \Delta_- = (\C_-, \delta_-)$ be two buildings of the same type $(W, S)$. A \textit{codistance} (or a \textit{twinning}) between $\Delta_+$ and $\Delta_-$ is a mapping $\delta_* : \left( \C_+ \times \C_- \right) \cup \left( \C_- \times \C_+ \right) \to W$ satisfying the following axioms, where $\epsilon \in \{+,-\}, x\in \Cp, y\in \Cm$ and $w=\delta_*(x, y)$:
\begin{enumerate}[label=(Tw\arabic*)]
	\item $\delta_*(y, x) = w^{-1}$;
	
	\item if $z\in \Cm$ is such that $s := \delta_{-\epsilon}(y, z) \in S$ and $\ell(ws) = \ell(w) -1$, then $\delta_*(x, z) = ws$;
	
	\item if $s\in S$, there exists $z\in \Cm$ such that $\delta_{-\epsilon}(y, z) = s$ and $\delta_*(x, z) = ws$.
\end{enumerate}
A \textit{twin building of type $(W, S)$} is a triple $\Delta = (\Delta_+, \Delta_-, \delta_*)$ where $\Delta_+ = (\C_+, \delta_+)$, $\Delta_- = (\C_-, \delta_-)$ are buildings of type $(W, S)$ and where $\delta_*$ is a twinning between $\Delta_+$ and $\Delta_-$.

We put $\C := \C_+ \cup \C_-$. Let $\epsilon \in \{+,-\}$. For $x\in \Cp$ we put $x^{\mathrm{op}} := \{ y\in \Cm \mid \delta_*(x, y) = 1_W \}$. It is a direct consequence of (Tw$1$) that $y\in x^{\mathrm{op}}$ if and only if $x\in y^{\mathrm{op}}$ for any pair $(x, y) \in \Cp \times \Cm$. If $y\in x^{\mathrm{op}}$ then we say that $y$ is \textit{opposite} to $x$ or that \textit{$(x, y)$ is a pair of opposite chambers}.

A \textit{residue} (resp.\ \textit{panel}) of $\Delta$ is a residue (resp.\ panel) of $\Delta_+$ or $\Delta_-$; given a residue $R\subseteq \C$ then we define its type and rank as before. The twin building $\Delta$ is called \textit{thick} if $\Delta_+$ and $\Delta_-$ are thick.

Let $\Sigma_+ \subseteq \C_+$ and $\Sigma_- \subseteq \C_-$ be apartments of $\Delta_+$ and $\Delta_-$, respectively. Then the set $\Sigma := \Sigma_+ \cup \Sigma_-$ is called \textit{twin apartment} if $\vert x^{\mathrm{op}} \cap \Sigma \vert = 1$ holds for each $x\in \Sigma$. If $(x, y)$ is a pair of opposite chambers, then there exists a unique twin apartment containing $x$ and $y$. We will denote it by $A(x, y)$. For more information we refer to \cite[Section $5.8.4$]{AB08}.

An \textit{automorphism} of $\Delta$ is a bijection $\phi: \C \to \C$ such that $\phi$ preserves the sign, the distance functions $\delta_{\epsilon}$ and the codistance $\delta_*$.

\subsection*{Root group data}

An \textit{RGD-system of type $(W, S)$} is a pair $\mathcal{D} = \left( G, \left( U_{\alpha} \right)_{\alpha \in \Phi}\right)$ consisting of a group $G$ together with a family of subgroups $U_{\alpha}$ (called \textit{root groups}) indexed by the set of roots $\Phi$, which satisfies the following axioms, where $H := \bigcap_{\alpha \in \Phi} N_G(U_{\alpha})$ and $U_{\epsilon} := \langle U_{\alpha} \mid \alpha \in \Phi_{\epsilon} \rangle$ for $\epsilon \in \{+, -\}$:
\begin{enumerate}[label=(RGD\arabic*)] \setcounter{enumi}{-1}
	\item For each $\alpha \in \Phi$, we have $U_{\alpha} \neq \{1\}$.
	
	\item For each prenilpotent pair $\{ \alpha, \beta \} \subseteq \Phi$ with $\alpha \neq \beta$, the commutator group $[U_{\alpha}, U_{\beta}]$ is contained in the group $U_{(\alpha, \beta)} := \langle U_{\gamma} \mid \gamma \in (\alpha, \beta) \rangle$.
	
	\item For every $s\in S$ and each $u\in U_{\alpha_s} \backslash \{1\}$, there exist $u', u'' \in U_{-\alpha_s}$ such that the product $m(u) := u' u u''$ conjugates $U_{\beta}$ onto $U_{s\beta}$ for each $\beta \in \Phi$.
	
	\item For each $s\in S$, the group $U_{-\alpha_s}$ is not contained in $U_+$.
	
	\item $G = H \langle U_{\alpha} \mid \alpha \in \Phi \rangle$.
\end{enumerate}
For $w\in W$ we define $U_w := \langle U_{\alpha} \mid w \notin \alpha \in \Phi_+ \rangle$. Let $G \in \mathrm{Min}(w)$ and let $(\alpha_1, \ldots, \alpha_k)$ be the sequence of roots crossed by $G$. Then we have $U_w = U_{\alpha_1} \cdots U_{\alpha_k}$ (cf.\ \cite[Corollary~$8.34(1)$]{AB08}). An RGD-system $\mathcal{D} = (G, (U_{\alpha})_{\alpha \in \Phi})$ is said to be \textit{over $\FF_2$} if every root group has cardinality $2$. In this case we denote for $\alpha \in \Phi$ the non-trivial element in $U_{\alpha}$ be $u_{\alpha}$.

Let $\mathcal{D} = (G, (U_{\alpha})_{\alpha \in \Phi})$ be an RGD-system of type $(W, S)$ and let $H = \bigcap_{\alpha \in \Phi} N_G(U_{\alpha})$, $B_{\epsilon} = H \langle U_{\alpha} \mid \alpha \in \Phi_{\epsilon} \rangle$ for $\epsilon \in \{+, -\}$. It follows from \cite[Theorem $8.80$]{AB08} that there exists an \textit{associated} twin building $\Delta(\mathcal{D}) = (\Delta(\mathcal{D})_+, \Delta(\mathcal{D})_-, \delta_*)$ of type $(W, S)$ such that $\Delta(\mathcal{D})_{\epsilon} = ( G/B_{\epsilon}, \delta_{\epsilon} )$ for $\epsilon \in \{ +, - \}$ and $G$ acts on $\Delta(\mathcal{D})$ by multiplication from the left. There is a distinguished pair of opposite chambers in $\Delta(\mathcal{D})$ corresponding to the subgroups $B_{\epsilon}$ for $\epsilon \in \{+, -\}$. We will denote this pair by $(c_+, c_-)$.

\begin{example}\label{exampleKM444}
	Let $\mathcal{D} = (\mathcal{G}, (U_{\alpha})_{\alpha \in \Phi})$ be the RGD-system associated with the split Kac--Moody group of type $(4, 4, 4)$ over $\FF_2$ (for the definition of Kac--Moody groups we refer to \cite{Ti87}). Then $\mathcal{D}$ is over $\FF_2$. Let $\{ \alpha, \beta \}$ be a prenilpotent pair. We will determine the commutator relations $[u_{\alpha}, u_{\beta}] \leq \langle U_{\gamma} \mid \gamma \in (\alpha, \beta) \rangle$. For $o(r_{\alpha} r_{\beta}) < \infty$, the commutator relations follow from \cite[Example $3.4$]{BiRGD}. For $o(r_{\alpha} r_{\beta}) = \infty$ we use the functoriality of Kac--Moody groups: Let $( \mathcal{G}, (\phi_i)_{i\in I}, \eta )$ be the system as in \cite[Ch.\ $2$]{Ti87}. For every field $\KK$ we let $U_{\alpha_i}(\KK) := \phi_i\left( \left\{ \begin{pmatrix}
		1 & k \\ 0 & 1
	\end{pmatrix} \mid k \in \KK \right\} \right)$ and $U_{-\alpha_i}(\KK) := \phi_i\left( \left\{ \begin{pmatrix}
		1 & 0 \\ k & 1
	\end{pmatrix} \mid k\in \KK \right\} \right)$ be the root groups corresponding to the simple roots. For every $i$ and any two fields $\FF$ and $\KK$ with a homomorphism $f: \FF \to \KK$ the following diagram commutes:
	\begin{center}
		\begin{tikzcd}
			\mathrm{SL}_2(\FF) \arrow[r, "\mathrm{SL}_2(f)"] \arrow[d, "\phi_i"] & \mathrm{SL}_2(\KK) \arrow[d, "\phi_i"] \\
			\mathcal{G}(\FF) \arrow[r, "\mathcal{G}(f)"] & \mathcal{G}(\KK)
		\end{tikzcd}
	\end{center}
	In particular, we have $\mathcal{G}(f)(U_{\alpha_i}(\FF)) \leq U_{\alpha_i}(\KK)$ and hence $\mathcal{G}(f)(U_{\gamma}(\FF)) \leq U_{\gamma}(\KK)$ for each root $\gamma \in \Phi$ by using (RGD$2$). Moreover, if $f$ is injective, then $\mathcal{G}(f)$ is injective by the axiom (KMG$4$) (cf.\ \cite{Ti87}). Let $f: \FF_2 \to \FF_4$ be the canonical inclusion. We have $[U_{\alpha}(\FF_4), U_{\beta}(\FF_4)] = 1$ by \cite[Theorem A]{Bi22}. This implies $\mathcal{G}(f)\left( [U_{\alpha}(\FF_2), U_{\beta}(\FF_2)] \right) \leq [U_{\alpha}(\FF_4), U_{\beta}(\FF_4)] = 1$ and, as $\mathcal{G}(f)$ is injective, we deduce $[U_{\alpha}(\FF_2), U_{\beta}(\FF_2)] = 1$. All in all we have the following commutator relations, where $U_{\gamma} = \langle u_{\gamma} \rangle$ for all $\gamma \in \Phi$:
	\begin{align*}
		[u_{\alpha}, u_{\beta}] = \begin{cases}
			\prod_{\gamma \in (\alpha, \beta)} u_{\gamma} & \text{if } o(r_{\alpha} r_{\beta}) < \infty, \vert (\alpha, \beta) \vert = 2 \\
			1 & \text{else.}
		\end{cases}
	\end{align*}
\end{example}

\section{RGD-systems of type $(4, 4, 4)$ over $\FF_2$}\label{Section: Action on building}

In this section we let $\mathcal{D} = (G, (U_{\alpha})_{\alpha \in \Phi})$ be an RGD-system of type $(4, 4, 4)$ over $\FF_2$. Furthermore, we let $V_{r_{\{s, t\}}} := \langle U_s, U_t \rangle \leq U_{r_{\{s, t\}}}$ for all $s \neq t \in S$. By \cite[Example $3.4$]{BiRGD} and \cite[Corollary $8.34(1)$]{AB08} this subgroup has index $2$ in $U_{r_{\{s, t\}}}$. Moreover, we let $\Delta(\mathcal{D}) = (\Delta(\mathcal{D})_+, \Delta(\mathcal{D})_-, \delta_*)$ be the twin building associated with $\mathcal{D}$ and let $(c_+, c_-)$ be the distinguished pair of opposite chambers. We denote for every $s\in S$ the unique chamber contained in $A(c_+, c_-)$ which is $s$-adjacent to $c_-$ by $c_s$. Then $U_+$ acts on $\Delta_- := \Delta(\mathcal{D})_-$. We abbreviate $c := c_-$ (for more information we refer to \cite[Section $8.9$]{AB08}).

\begin{remark}
	Note that $V_{r_{\{s, t\}}} = \{ 1, u_s, u_t, u_s u_t, u_t u_s, u_s u_t u_s, u_t u_s u_t, u_s u_t u_s u_t = u_t u_s u_t u_s \}$.
\end{remark}

\begin{lemma}\label{Lemma: Uplus acts on Deltaminus}
	The following hold:
	\begin{enumerate}[label=(\alph*)]
		\item $\ell(c_s, c_s.h) \geq 3$ for all $h \in V_{r_{\{s, t\}}} \backslash \{ 1, u_s \}$;
		
		\item $\ell(c_s, c_t.h) \geq 2$ for all $h \in V_{r_{\{s, t\}}}$;
		
		\item $\ell(c_t.h, p) \geq 2$ for all $p \in \P_t(c)$ and $h \in V_{r_{\{s, t\}}} \backslash \{ 1, u_t \}$;
		
		\item $\ell(c_s.h, p) \geq 2$ or $\delta(c_s.h, p) = s$ for all $p \in \P_t(c)$ and $h\in V_{r_{\{s, t\}}}$;
		
		\item $\ell(p, q) \geq 2$ or $\delta(p, q) = s$ for all $p \in \P_t(c.h), q \in \P_t(c)$ and $h\in V_{r_{\{s, t\}}} \backslash \{1, u_t\}$.
		
		\item $\ell(p, c_s) \geq 2$ or $\delta(p, c_s) = s$ for all $p \in \P_t(c.h)$ and $h\in V_{r_{\{s, t\}}}$.
	\end{enumerate}
\end{lemma}
\begin{proof}
	Before we prove the claim, we consider the following diagram, where the lower chambers are all opposite to $c_+$, the upper chambers $d$ satisfy $\ell(\delta_{\star}(c_+, d)) = 1$ and the letter in the triangles denotes the type of the panels:
	\begin{figure}[h]
		\begin{tikzpicture}[scale=0.95]
			\node[label=above:$c_t.u_su_tu_s$, circle, scale=0.5, fill=black] (t1sts) at ($(90:1)$) {};
			\node[label=below:$c.u_tu_su_tu_s$, circle, scale=0.5, fill=black] (tsts) at ($({120+90}:1)$) {};
			\node[label=below:$c.u_su_tu_s$, circle, scale=0.5, fill=black] (sts) at ($({2*120+90}:1)$) {};
			
			\node[label=above:$c_s.u_tu_s$, circle, scale=0.5, fill=black] (s1ts) at ($(90:1)+({2*cos(30)},0)$) {};
			\node[label=below:$c.u_tu_s$, circle, scale=0.5, fill=black] (ts) at ($({2*120+90}:1)+({2*cos(30)},0)$) {};
			
			\node[label=above:$c_t.u_s$, circle, scale=0.5, fill=black] (t1s) at ($(90:1)+({4*cos(30)},0)$) {};
			\node[label=below:$c.u_s$, circle, scale=0.5, fill=black] (s) at ($({2*120+90}:1)+({4*cos(30)},0)$) {};
			
			\node[label=above:$c_s$, circle, scale=0.5, fill=black] (s1) at ($(90:1)+({6*cos(30)},0)$) {};
			\node[label=below:$c$, circle, scale=0.5, fill=black] (c) at ($({2*120+90}:1)+({6*cos(30)},0)$) {};
			
			\node[label=above:$c_t$, circle, scale=0.5, fill=black] (t1) at ($(90:1)+({8*cos(30)},0)$) {};
			\node[label=below:$c.u_t$, circle, scale=0.5, fill=black] (t) at ($({2*120+90}:1)+({8*cos(30)},0)$) {};
			
			\node[label=above:$c_s.u_t$, circle, scale=0.5, fill=black] (s1t) at ($(90:1)+({10*cos(30)},0)$) {};
			\node[label=below:$c.u_su_t$, circle, scale=0.5, fill=black] (st) at ($({2*120+90}:1)+({10*cos(30)},0)$) {};
			
			\node[label=above:$c_t.u_su_t$, circle, scale=0.5, fill=black] (t1st) at ($(90:1)+({12*cos(30)},0)$) {};
			\node[label=below:$c.u_tu_su_t$, circle, scale=0.5, fill=black] (tst) at ($({2*120+90}:1)+({12*cos(30)},0)$) {};
			
			\node[label=above:$c_s.u_tu_su_t$, circle, scale=0.5, fill=black] (s1tst) at ($(90:1)+({14*cos(30)},0)$) {};
			\node[label=below:$c.u_su_tu_su_t$, circle, scale=0.5, fill=black] (stst) at ($({2*120+90}:1)+({14*cos(30)},0)$) {};
				
			\node at (0,0) {$t$};
			\node at ({2*cos(30)},0) {$s$};
			\node at ({4*cos(30)},0) {$t$};
			\node at ({6*cos(30)},0) {$s$};
			\node at ({8*cos(30)},0) {$t$};
			\node at ({10*cos(30)},0) {$s$};
			\node at ({12*cos(30)},0) {$t$};
			\node at ({14*cos(30)},0) {$s$};
			
			\draw (tsts) -- (sts);
			\draw (sts) -- (t1sts);
			\draw (t1sts) -- (tsts);
			
			\draw (sts) -- (ts);
			\draw (ts) -- (s1ts);
			\draw (s1ts) -- (sts);
			
			\draw (ts) -- (s);
			\draw (s) -- (t1s);
			\draw (t1s) -- (ts);
			
			\draw (s) -- (c);
			\draw (c) -- (s1);
			\draw (s1) -- (s);
			
			\draw (c) -- (t);
			\draw (t) -- (t1);
			\draw (t1) -- (c);
			
			\draw (t) -- (st);
			\draw (st) -- (s1t);
			\draw (s1t) -- (t);
			
			\draw (st) -- (tst);
			\draw (tst) -- (t1st);
			\draw (t1st) -- (st);
			
			\draw (tst) -- (stst);
			\draw (stst) -- (s1tst);
			\draw (s1tst) -- (tst);			
		\end{tikzpicture}
	\end{figure}
	
	We first show Assertion $(a)$. As $c_s = c_s.u_s$ and $u_tu_su_tu_s = u_su_tu_su_t$, we can assume $h \in \{ u_t, u_t u_s, u_t u_s u_t \}$. Now we deduce the following:
	\begin{enumerate}[label=(\roman*)]
		\item $\ell(c_s, c_s.u_t) = 3$;
		
		\item $\ell(c_s, c_s.u_tu_s) = \ell(c_s, c_s.u_t) = 3$;
		
		\item $\ell(c_s, c_s.u_t u_s u_t) \geq 3$.
	\end{enumerate}
	To show Assertion $(b)$ we can assume $h \in \{ 1, u_s, u_su_t, u_su_tu_s \}$. We deduce the following:
	\begin{enumerate}[label=(\roman*)]
		\item $\ell(c_s, c_t) = 2$;
		
		\item $\ell(c_s, c_t.u_s) = \ell(c_s, c_t) = 2$;
		
		\item $\ell(c_s, c_t.u_s u_t) = 4$;
		
		\item $\ell(c_s, c_t.u_s u_t u_s) = \ell(c_s, c_t. u_s u_t) = 4$
	\end{enumerate}
	For Assertion $(c)$ we can assume $h \in \{ u_s, u_s u_t, u_s u_t u_s \}$. We deduce the following:
	\begin{enumerate}[label=(\roman*)]
		\item $\ell(c_t.u_s, p) \in \{2, 3\}$;
		
		\item $\ell(c_t.u_su_t, p) \in \{2, 3\}$;
		
		\item $\ell(c_t.u_s u_t u_s, p) \in \{3, 4\}$;
	\end{enumerate}
	For Assertion $(d)$ we can assume $h\in \{ 1, u_t, u_t u_s, u_t u_s u_t \}$. We deduce the following:
	\begin{enumerate}[label=(\roman*)]
		\item $\delta(c_s, p) \in \{s, st\}$;
		
		\item $\delta(c_s.u_t, p) \in \{s, st\}$;
		
		\item $\ell(c_s.u_tu_s, p) \geq 3$;
		
		\item $\ell(c_s.u_t u_s u_t, p) \geq 3$.
	\end{enumerate}
	For Assertion $(e)$ we can assume $h \in \{ u_s, u_s u_t, u_s u_t u_s \}$, as $\P_t(c.u_t) = \P_t(c)$. We deduce the following:
	\begin{enumerate}[label=(\roman*)]
		\item $h=u_s$: We have $\delta(p, q) = s$ or $\ell(p, q) \in \{ 2, 3 \}$.
		
		\item $h = u_s u_t$: This follows similar as in the case $h=u_s$.
		
		\item $h=u_s u_t u_s$: Then we have $\ell(p, q) \geq 3$.
	\end{enumerate}
	For Assertion $(f)$ we can assume $h \in \{ 1, u_s, u_su_t, u_s u_t u_s \}$. We deduce the following:
	\begin{enumerate}[label=(\roman*)]
		\item $h = 1$: Then $\delta(p, c_s) = s$ or $\ell(p, c_s) = 2$.
		
		\item $h = u_s$: Again we obtain $\delta(p, c_s) = s$ or $\ell(p, c_s) = 2$.
		
		\item $h = u_s u_t$: In this case we have $\ell(p, c_s) \geq 3$.
		
		\item $h = u_s u_t u_s$: Again we obtain $\ell(p, c_s) \geq 3$. \qedhere
	\end{enumerate}
\end{proof}

\begin{remark}
	For each root $\alpha \in \Phi_+$ there exist $w\in W, s\in S$ with $\alpha = w\alpha_s$. For short we will write $u_{ws}$ to be the generator of $U_{w\alpha_s}$.
\end{remark}

\begin{lemma}\label{Lemma: Normal form}
	Let $n>0$, let $g_1, \ldots, g_n \in \{ u_{sr}, u_{tr}, u_{rt}, u_{rt} u_{tr} \}$ and let $h_1, \ldots, h_n \in V_{r_{\{s, t\}}}$ be such that $h_i \notin \{ 1, u_s\}$ if $g_i = g_{i+1} = u_{sr}$ and such that $h_i \notin \{ 1, u_t \}$ if $g_i, g_{i+1} \in \{ u_{tr}, u_{rt}, u_{rt} u_{tr} \}$. Then $g_1 h_1 \cdots g_n h_n \neq 1$ holds in $G$.
\end{lemma}
\begin{proof}
	Note that $g_n^{-1} = g_n$. For the proof we consider the action of $U_+$ on $\Delta_-$ as in Lemma \ref{Lemma: Uplus acts on Deltaminus}. We abbreviate $\delta := \delta_-, c := B_-, R_{ef} := R_{\{e, f\}}(c)$ for any $e\neq f \in S$ and we let $g := g_1 h_1 \cdots g_n h_n$. We show the following by induction on $n \geq 1$:
	\allowdisplaybreaks
	\begin{itemize}
		\item If $g_n = u_{fr}$ for some $f\in \{s, t\}$, then the following hold:
		\begin{enumerate}[label=(\alph*)]
			\item $\proj_{R_{st}} (c.g) = c_f.h_n$;
			
			\item $\ell(c.g, \proj_{R_{st}} (c.g)) >0$.
		\end{enumerate}
		
		\item If $g_n \in \{ u_{rt}, u_{rt} u_{tr} \}$, then the following hold:
		\begin{enumerate}[label=(\alph*)]
			\item $\proj_{R_{st}} (c.g) = \proj_{\P_t(c.h_n)} (c.g)$;
			
			\item $\ell(\delta(c.g, \proj_{R_{st}} (c.g))srs) = \ell(c.g, \proj_{R_{st}} (c.g)) +3$;
			
			\item $\ell(c.g, \proj_{R_{st}} (c.g)) >0$.
		\end{enumerate}
	\end{itemize}
	Once this is shown, the claim follows, since $g=1$ would imply $\ell(c.g, \proj_{R_{st}} (c.g)) = 0$. Let $n=1$ and suppose $g_1 \in \{ u_{sr}, u_{tr} \}$. Then we have $\proj_{R_{st}} (c.u_{fr}) = c_f$ by \cite[Lemma $2.15$]{AB08} and, in particular, $\proj_{R_{st}} (c.g) = \left( \proj_{R_{st}} (c.g_1) \right).h_1 = c_f.h_1$. Moreover, we have $\ell(c.g, \proj_{R_{st}}(c.g)) = \ell(c.g, c_f.h_1) = \ell(c.g_1, c_f) >0$. Now we suppose $g_1 \in \{ u_{rt}, u_{rt} u_{tr} \}$. Note that $\delta(c, c.u_{rt} u_{tr}) = \delta(c, c.u_r u_t u_r u_t) = r_{\{r, t\}}$ and $\delta(c, c.u_{rt}) = rtr$. In particular, we have $\delta(c.g_1, c) \in \{ rtr, r_{\{r, t\}} \}$. Let $q = \proj_{\P_t(c)} (c.g_1)$. Then $\delta(c.g_1, q) = rtr$ and, by \cite[Lemma $2.15$]{AB08}, $\ell(\delta(c.g_1, q)s) = \ell(\delta(c.g_1, q)) +1$. Hence $q = \proj_{R_{st}} (c.g_1)$. This implies $\proj_{R_{st}} (c.g) = q.h_1 = \proj_{\P_t(c.h_1)} (c.g)$. Since $\delta(c.g_1, q) = rtr$, we infer (again by \cite[Lemma $2.15$]{AB08}) $\proj_{\P_r(q)} (c.g_1) = \proj_{R_{\{s, r\}}(q)} (c.g_1)$. Thus we have $\delta(\proj_{R_{\{r, s\}}(q)} (c.g_1), q) =r$ and hence $\ell(\delta(c.g, q.h_1)srs) = \ell(\delta(c.g_1, q)srs) = \ell(c.g_1, q) +3 = \ell(c.g, q.h_1) +3$. Moreover, $\ell(c.g, \proj_{R_{st}}(c.g)) = \ell(c.g_1, q) >0$.
	
	Now we assume $n>1$. We define $h := g_1 h_1 \cdots g_{n-1} h_{n-1}$. In both cases we will show $\ell(c.h, \proj_{R_{st}.g_n} (c.h)) > \ell(c.h, \proj_{R_{st}} (c.h))$ instead of $\ell(c.g, \proj_{R_{st}} (c.g)) >0$, because it implies $\ell(c.g, \proj_{R_{st}} (c.g)) = \ell(c.h, \proj_{R_{st}.g_n} (c.h)) > \ell(c.h, \proj_{R_{st}} (c.h))>0$ by induction. We distinguish the following cases, where the first case is a special case which we will use in the other cases:
	\begin{enumerate}[label=(\alph*)]
		\item\label{Case1} $g_n \in \{ u_{rt}, u_{rt} u_{tr} \}$ and $\proj_{\P_t(c)} (c.h) = \proj_{R_{rt}} (c.h)$: As above, we deduce $\delta(c, c.g_n) \in \{rtr, r_{\{r, t\}} \}$. We define $p := \proj_{\P_t(c)} (c.h) = \proj_{R_{rt}} (c.h)$. As $p \in \P_t(c)$, we deduce $\delta(p, c.g_n) \in \{ rtr, r_{\{r, t\}} \}$. We define $q := \proj_{\P_t(c.g_n)} (c.h)$ and note that $q \in R_{rt}$. Then $q = \proj_{\P_t(c.g_n)} \proj_{R_{rt}} (c.h)$, $\delta(p, q) = rtr$ and Lemma \ref{wordsincoxetergroup} implies $\ell(\delta(c.h, q)s) = \ell(c.h, q) +1$. Thus $q = \proj_{R_{\{s, t\}}(c.g_n)} (c.h) = \proj_{R_{st}.g_n} (c.h)$ and hence $\proj_{\P_t(c.h_n)} (c.g) = q.g_nh_n = \proj_{R_{st}} (c.g)$. Since $\delta(p, q) = rtr$ and $p \in \P_t(c) \subseteq R_{st}$, we deduce
		\allowdisplaybreaks
		\begin{align*}
			\ell(c.h, \proj_{R_{st}.g_n} (c.h)) = \ell(c.h, q) \overset{q\in R_{rt}}{=} \ell(c.h, p) +3 \overset{p \in R_{st}}{>} \ell(c.h, \proj_{R_{st}} (c.h)).
		\end{align*}
		Moreover, as $\ell(p, \proj_{\P_r(q)} (c.h)) = \ell(p, \proj_{\P_r(q)} p) = 2$, Lemma \ref{wordsincoxetergroup} implies that $\proj_{\P_r(q)} (c.h) = \proj_{R_{\{r, s\}}(q)} (c.h)$ and hence $\ell(\delta(c.g, q.g_nh_n)srs) = \ell(\delta(c.h, q)srs) = \ell(c.h, q) +3 = \ell(c.g, q.g_nh_n) +3$.
		
		\item $g_{n-1} = u_{fr}$ for some $f \in \{s, t\}$: Then we have $\proj_{R_{st}} (c.h) = c_f.h_{n-1}$ by induction. We distinguish the following two cases:
		\begin{enumerate}[label=(\roman*)]
			\item $g_n \notin \{ u_{rt}, u_{rt} u_{tr} \}$: Then there exists $e\in \{s, t\}$ with $g_n = u_{er}$. If $e=f$, we have $h_{n-1} \notin \{ 1, u_f \}$ by assumption and $\ell(c_f.h_{n-1}, c_e) \geq 3$ by Lemma \ref{Lemma: Uplus acts on Deltaminus}$(a)$. If $e \neq f$, we have $\ell(c_f.h_{n-1}, c_e) \geq 2$ by Lemma \ref{Lemma: Uplus acts on Deltaminus}$(b)$. Note that in both cases we have $\delta(c_f.h_{n-1}, c_e) \in \langle s, t \rangle$. Using Lemma \ref{wordsincoxetergroup} we obtain $\ell(\delta(c.h, c_e)ru) = \ell(\delta(c.h, c_f.h_{n-1}) \delta(c_f.h_{n-1}, c_e) ru ) = \ell(c.h, c_f.h_{n-1}) + \ell(c_f.h_{n-1}, c_e) +2 = \ell(c.h, c_e) +2$ for each $u \in \{s, t\}$. Since $\delta(c_e, c_e.u_{er}) = r$, the previous computations imply that $c_e.u_{er} = \proj_{R_{\{s, t\}}(c_e.u_{er})} (c.h) = \proj_{R_{st}.u_{er}} (c.h)$ and hence $c_e.h_n = \proj_{R_{st}} (c.g)$. In particular, we have $\ell(c.h, \proj_{R_{st}.g_n} (c.h)) = \ell(c.h, c_e.u_{er}) = \ell(c.h, c_e) +1 > \ell(c.h, c_f.h_{n-1}) = \ell(c.h, \proj_{R_{st}} (c.h))$.
			
			\item $g_n \in \{ u_{rt}, u_{rt} u_{tr} \}$: We define $p := \proj_{\P_t(c)} (c.h)$. If $f=t$, we have $h_{n-1} \notin \{1, u_t\}$ and hence $\ell(c_t.h_{n-1}, p) \geq 2$ by Lemma \ref{Lemma: Uplus acts on Deltaminus}$(c)$. By Lemma \ref{wordsincoxetergroup} we obtain that $\ell(\delta(c.h, p)r) = \ell(c.h, p) +1$ and hence $p = \proj_{R_{rt}} (c.h)$. The claim follows now from Case \ref{Case1}. If $f=s$, Lemma \ref{Lemma: Uplus acts on Deltaminus}$(d)$ yields $\ell(c_s.h_{n-1}, p) \geq 2$ or $\delta(c_s.h_{n-1}, p) = s$. If $\ell(c_s.h_{n-1}, p) \geq 2$, we obtain $p = \proj_{R_{rt}} (c.h)$ as before and the claim follows again from Case \ref{Case1}. Thus we suppose that $\delta(c_s.h_{n-1}, p) = s$. Note that we have $\delta(c, c.u_{rt} u_{tr}) = r_{\{r, t\}}$ and $\delta(c, c.u_{rt}) = rtr$. In particular, we have $\delta(p, c.g_n) \in \{ rtr, r_{\{r, t\}} \}$. If $\ell(\delta(c.h, p)r) = \ell(c.h, p) +1$, we have $p = \proj_{R_{rt}} (c.h)$ and the claim follows as before. Thus we assume $\ell(\delta(c.h, p)r) = \ell(c.h, p) -1$. Then $\ell(\delta(c.h, p)rt) = \ell(c.h, p)$ by Lemma \ref{Lemma: not both down} and hence $\ell(wu) = \ell(w) +1$ for $w = \delta(c.h, p)r$ and each $u\in \{r, t\}$. Since $p \in \P_t(c) \subseteq R_{rt}$ and hence $\P_r(p) \subseteq R_{rt}$, we infer $\proj_{\P_r(p)} (c.h) = \proj_{R_{rt}} (c.h)$. We define $q := \proj_{\P_t(c.g_n)} (c.h)$. Since $\delta(p, c.g_n) \in \{ rtr, r_{\{r, t\}} \}$ we have $\ell(\proj_{R_{rt}} (c.h), q) \geq 2$ and hence $\ell(c.h, q) = \ell(c.h, \proj_{R_{rt}} (c.h)) +2 > \ell(c.h, \proj_{R_{rt}} (c.h)) +1 = \ell(c.h, p)$. Lemma \ref{wordsincoxetergroup} implies $q = \proj_{R_{st}(c.g_n)} (c.h) = \proj_{R_{st}.g_n} (c.h)$ and hence, as $p \in R_{st}$, we deduce
			\allowdisplaybreaks
			\begin{align*}
				\ell(c.h, \proj_{R_{st}.g_n} (c.h)) = \ell(c.h, q) > \ell(c.h, p) > \ell(c.h, \proj_{R_{st}} (c.h)).
			\end{align*}
			Moreover, we have $\proj_{\P_t(c.h_n)} (c.g) = q.g_nh_n = \proj_{R_{st}} (c.g)$. We have already mentioned that $\ell(\proj_{R_{rt}} (c.h), q) \geq 2$. Using Lemma \ref{Lemma: not both down}, Lemma \ref{wordsincoxetergroup} and the fact that $\ell(\delta(c.h, \proj_{R_{rt}} (c.h))rs) = \ell(c.h, \proj_{R_{rt}} (c.h))$, we deduce $\ell(\delta(c.h, z)s) = \ell(c.h, z)+1$ for all $z \in R_{rt} \backslash \P_r(p)$. Note that $\P_r( \proj_{R_{rt}} (c.h) ) = \P_r(p)$. In particular, as $\delta(\proj_{R_{rt}} (c.h), \proj_{\P_r(q)} (c.h)) \in \{t, rt\}$, we deduce $\proj_{\P_r(q)} (c.h) \in R_{rt} \backslash \P_r(p)$. Thus we have $\proj_{R_{\{r, s\}}(q)} (c.h) = \proj_{\P_r(q)} (c.h)$ and hence $\ell(\delta(c.g, q.g_nh_n)srs) = \ell(\delta(c.h, q)srs) = \ell(c.h, q) +3 = \ell(c.g, qg_nh_n) +3$.
		\end{enumerate}
		
		\item $g_{n-1} \in \{ u_{rt}, u_{rt} u_{tr} \}$: We define $p := \proj_{R_{st}} (c.h)$. Using induction, we have $p = \proj_{\P_t(c.h_{n-1})} (c.h)$ and $\ell(\delta(c.h, p)srs) = \ell(c.h, p) +3$. We distinguish the following three cases:
		\begin{enumerate}[label=(\roman*)]
			\item $g_n = u_{tr}$: Then we have $h_{n-1} \notin \{ 1, u_t \}$ by assumption. As $h_{n-1}^{-1} \notin \{1, u_t\}$ and $p.h_{n-1}^{-1} \in \P_t(c)$, Lemma \ref{Lemma: Uplus acts on Deltaminus}$(c)$ yields $\ell(p, c_t) = \ell(p.h_{n-1}^{-1}, c_t.h_{n-1}^{-1}) \geq 2$. Using Lemma \ref{wordsincoxetergroup} we obtain $\ell(\delta(c.h, c_t)ru) = \ell(c.h, c_t) +2$ for each $u\in \{s, t\}$ and hence $c_t.u_{tr} = \proj_{R_{st}.u_{tr}} (c.h)$, as $\delta(c_t, c_t.u_{tr}) = r$. This implies $c_t.h_n = \proj_{R_{st}} (c.g)$. Moreover, we have $\ell(c.h, \proj_{R_{st}.g_n} (c.h)) = \ell(c.h, c_t.g_n) > \ell(c.h, p) = \ell(c.h, \proj_{R_{st}} (c.h))$.
			
			\item $g_n \in \{ u_{rt}, u_{rt} u_{tr} \}$: Then we have $h_{n-1} \notin \{ 1, u_t \}$ by assumption. We define $q := \proj_{\P_t(c)} (c.h)$. Using Lemma \ref{Lemma: Uplus acts on Deltaminus}$(e)$ we have either $\ell(p, q) \geq 2$ or $\delta(p, q) = s$. If $\ell(p, q) \geq 2$, we obtain $q = \proj_{R_{rt}} (c.h)$ by Lemma \ref{wordsincoxetergroup}. Now suppose $\delta(p, q) = s$. Note that we have $\ell(\delta(c.h, p)srs) = \ell(c.h, p) +3$ by induction and hence $\ell(\delta(c.h, q)r) = \ell(c.h, q) +1$. In particular, $q = \proj_{R_{rt}} (c.h)$. Both cases yield $q = \proj_{R_{rt}} (c.h)$ and the claim follows from Case \ref{Case1}.
			
			\item $g_n = u_{sr}$: Using Lemma \ref{Lemma: Uplus acts on Deltaminus}$(f)$ we have either $\ell(p, c_s) \geq 2$ or $\delta(p, c_s) = s$. If $\ell(p, c_s) \geq 2$, we obtain $\ell(\delta(c.h, c_s)ru) = \ell(c.h, c_s) +2$ for each $u \in \{s, t\}$ by Lemma \ref{wordsincoxetergroup} and hence $c_s.u_{sr} = \proj_{R_{st}.u_{sr}} (c.h)$. This implies $c_s.h_n = \proj_{R_{st}} (c.g)$ as well as $\ell(c.h, \proj_{R_{st}.g_n} (c.h)) = \ell(c.h, c_s.g_n) > \ell(c.h, p) = \ell(c.h, \proj_{R_{st}} (c.h))$.
			
			Suppose now that $\delta(p, c_s) = s$. By induction we have $\ell(\delta(c.h, p)srs) = \ell(c.h, p) +3$. By Lemma \ref{wordsincoxetergroup} we obtain $\ell(\delta(c.h, p)srt) = \ell(c.h, p) +3$. Since $\delta(p, c_s) = s$ and $\delta(c_s, c_s.u_{sr}) = r$, we have $\delta(p, c_s.u_{sr}) = sr$ and $c_s.u_{sr} = \proj_{R_{st}.u_{sr}} (c.h)$. This implies $c_s.h_n = \proj_{R_{st}} (c.g)$ and, in particular, that $\ell(c.h, \proj_{R_{st}.g_n} (c.h)) = \ell(c.h, c_s.g_n) > \ell(c.h, p) = \ell(c.h, \proj_{R_{st}} (c.h))$. \qedhere
		\end{enumerate}
	\end{enumerate}
\end{proof}

\begin{remark}
	Let $G_1, G_2, G_3, H_1, H_2$ be groups and let $\alpha_i: H_i \to G_i$ and $\omega_i: H_i \to G_{i+1}$ be monomorphisms for all $i \in \{1, 2\}$. Then we denote by $G_1 \star_{H_1} G_2 \star_{H_2} G_3$ the corresponding tree product. For more information we refer to Section \ref{Subsection: Graph of groups}.
\end{remark}

\begin{theorem}\label{theoremsubgroupKM}
	The canonical homomorphism $\phi: U_{sr} \star_{U_s} V_{r_{\{s, t\}}} \star_{U_t} U_{trt} \to G$ is injective.
\end{theorem}
\begin{proof}
	We abbreviate $H := U_{sr} \star_{U_s} V_{r_{\{s, t\}}} \star_{U_t} U_{trt}$. We note that any $g\in H$ can be written in the form $h_0 g_1 h_1 \cdots g_n h_n$, where $g_i \in \{ u_{sr}, u_{tr}, u_{rt}, u_{rt} u_{tr} \}, h_i \in V_{r_{\{s, t\}}}$ and $n\geq 0$. We reduce the product as follows:
	\begin{enumerate}[label=(\alph*)]
		\item Suppose that $g_i = g_{i+1} = u_{sr}$ and $h_i \in \{ 1, u_s \}$ for some $1 \leq i \leq n-1$. Then $g_i h_i g_{i+1} = h_i$, as $[g_i, h_i] = 1$. Thus
		\allowdisplaybreaks
		\begin{align*}
			g = h_0 g_1 h_1 \cdots g_{i-1} (h_{i-1} h_i h_{i+1}) g_{i+2} h_{h+2} \cdots g_n h_n  
		\end{align*}
		
		\item Suppose that $g_i, g_{i+1} \in \{ u_{tr}, u_{rt}, u_{rt} u_{tr} \}$ and $h_i \in \{1, u_t\}$ for some $1 \leq i \leq n-1$. Then $g_i h_i g_{i+1} = h_i g_i g_{i+1}$, as $[g_i, h_i] = 1$. We distinguish the following two cases:
		\begin{enumerate}[label=(\roman*)]
			\item $g_i = g_{i+1}$: Then we can write $g$ as before as
			\allowdisplaybreaks
			\begin{align*}
				g = h_0 g_1 h_1 \cdots g_{i-1} (h_{i-1} h_i h_{i+1}) g_{i+2} h_{i+2} \cdots g_n h_n
			\end{align*}
			
			\item $g_i \neq g_{i+1}$: Then $g_i g_{i+1} \in \{ u_{tr}, u_{rt}, u_{rt} u_{tr} \}$ and we can write $g$ as follows:
			\allowdisplaybreaks
			\begin{align*}
				g = h_0 g_1 h_1 \cdots g_{i-1} (h_{i-1} h_i) (g_i g_{i+1}) h_{i+1} \cdots g_n h_n
			\end{align*}
		\end{enumerate}
	\end{enumerate}
	In each step we reduce the length of the product. As we can only reduce finitely many times, we can not apply $(a)$ or $(b)$ at some point. In particular, any $g\in H$ can be written as $h_0 g_1 h_1 \cdots g_n h_n$, where $g_i \in \{ u_{sr}, u_{tr}, u_{rt}, u_{rt} u_{tr} \}$, $h_i \in V_{r_{\{s, t\}}}$ and if $g_i = g_{i+1} = u_{sr}$ for some $1 \leq i \leq n-1$, then $h_i \notin \{ 1, u_s \}$ and if $g_i, g_{i+1} \in \{ u_{tr}, u_{rt}, u_{rt} u_{tr} \}$ for some $1 \leq i \leq n-1$, then $h_i \notin \{1, u_t\}$.
	
	Assume that $\phi$ is not injective and let $1 \neq g \in \ker(\phi)$. Then there exist $g_i, h_i$ as before such that $g = h_0 g_1 h_1 \cdots g_nh_n$. As $V_{r_{\{s, t\}}} \cap \ker(\phi) = \{1\}$, we have $n>0$. As $\ker(\phi) \trianglelefteq H$, we have $g_1 h_1 \cdots g_n \left( h_nh_0 \right) = h_0^{-1} g h_0 \in \ker(\phi)$. But Lemma \ref{Lemma: Normal form} implies $g_1 h_1 \cdots g_n \left( h_nh_0 \right) \neq 1$ in $G$, which yields a contradiction. Thus $\phi$ is injective.
\end{proof}

\section{An application}\label{Section: Application}

\subsection*{Graphs of groups}\label{Subsection: Graph of groups}

This subsection is based on \cite[Section $2$]{KWM05} and \cite{Se79}.

Following Serre, a \textit{graph} $\Gamma$ consists of a vertex set $V\Gamma$, an edge set $E\Gamma$, the inverse function $^{-1}:E\Gamma \to E\Gamma$ and two edge endpoint functions $o: E\Gamma \to V\Gamma, t: E\Gamma \to V\Gamma$ satisfying the following axioms:
\begin{enumerate}[label=(\roman*)]
	\item The function $^{-1}$ is a fixed-point free involution on $E\Gamma$;
	
	\item For each $e\in E\Gamma$ we have $o(e) = t(e^{-1})$.
\end{enumerate}
A \textit{tree of groups} is a triple $\mathbb{G} = (T, (G_v)_{v\in V\Gamma}, (G_e)_{e\in E\Gamma})$ consisting of a finite tree $T$ (i.e.\ $VT$ and $ET$ are finite), a family of \textit{vertex groups} $(G_v)_{v\in VT}$ and a family of \textit{edge groups} $(G_e)_{e\in ET}$. Every edge $e \in ET$ comes equipped with two \textit{boundary monomorphisms} $\alpha_e: G_e \to G_{o(e)}$ and $\omega_e: G_e \to G_{t(e)}$. We assume that for each $e\in ET$ we have $G_{e^{-1}} = G_e$, $\alpha_{e^{-1}} = \omega_e$ and $\omega_{e^{-1}} = \alpha_e$. We let $G_T := \lim \mathbb{G}$ be the direct limit of the inductive system formed by the vertex groups, edge groups and boundary monomorphisms and call $G_T$ a \textit{tree product}. A \textit{sequence of groups} is a tree of groups where the underlying graph is a sequence. If the tree $T$ is a \emph{segment}, i.e.\ $VT = \{v, w\}$ and $ET = \{ e, e^{-1} \}$, then the tree product $G_T$ is an amalgamated product. We will use the notation from amalgamated products and we will write $G_T = G_v \star_{G_e} G_w$. We extend this notation to arbitrary \emph{sequences} $T$: if $VT = \{ v_0, \ldots, v_n \}$, $ET = \{ e_i, e_i^{-1} \mid 1 \leq i \leq n \}$ and $o(e_i) = v_{i-1}, t(e_i) = v_i$, then we will write $G_T = G_{v_0} \star_{G_{e_1}} G_{v_1} \star_{G_{e_2}} \cdots \star_{G_{e_n}} G_{v_n}$.

\begin{proposition}[{\cite[Theorem $1$]{KS70}}]\label{treeproducts}
	Let $\mathbb{G} = (T, (G_v)_{v\in VT}, (G_e)_{e\in ET})$ be a tree of groups. If $T$ is partitioned into subtrees whose tree products are $G_1, \ldots, G_n$ and the subtrees are contracted to vertices, then $G_T$ is isomorphic to the tree product of the tree of groups whose vertex groups are the $G_i$ and the edge groups are the $G_e$, where $e$ is the unique edge which joins two subtrees. Moreover, $G_i \to G_T$ is injective.
\end{proposition}

\begin{proposition}\label{treeofgroupsinjective}
	Let $T$ be a tree and let $T'$ be a subtree of $T$. Moreover, we let $\mathbb{G} = (T, (G_v)_{v\in VT}, (G_e)_{e\in ET})$ and $\mathbb{H} = (T', (H_v)_{v\in VT'}, (H_e)_{e\in ET'})$ be two trees of groups and suppose the following:
	\begin{enumerate}[label=(\roman*)]
		\item For each $v\in VT'$ we have $H_v \leq G_v$.
		
		\item\label{Case: ii} For each $e\in ET'$ we have $\alpha_e^{-1}(H_{o(e)}) = \omega_e^{-1}(H_{t(e)})$.
		
		\item For each $e\in ET'$ the group $H_e$ coincides with the group in \ref{Case: ii}.
	\end{enumerate}
	Then the canonical homomorphism $\nu: H_{T'} \to G_T$ between the tree product $H_{T'}$ and the tree product $G_T$ is injective. In particular, we have $\nu(H_{T'}) \cap G_v = H_v$ for each $v\in VT'$.
\end{proposition}
\begin{proof}
	This follows from \cite[Proposition~$4.3$]{KWM05} and \cite[Proposition~$20$]{Se79}.
\end{proof}

\begin{corollary}\label{intersectionwithasubtree}
	Let $\mathbb{G} = (T, (G_v)_{v\in VT}, (G_e)_{e\in ET})$ be a tree of groups and let $H_v \leq G_v$ for each $v\in VT$. Assume that $H_e := \alpha_e^{-1}(H_{o(e)}) = \omega_e^{-1}(H_{t(e)})$ for all $e \in ET$ and let $\mathbb{H} = (T, (H_v)_{v\in VT}, (H_e)_{e\in ET})$ be the associated tree of groups. Let $T'$ be a subtree of $T$ and let $\mathbb{L} = (T', (G_v)_{v\in VT'}$, $(G_e)_{e\in ET'})$, $\mathbb{K} = (T', (H_v)_{v\in VT'}, (H_e)_{e\in ET'})$. Then $H_T \cap L_{T'} = K_{T'}$ in $G_T$.
\end{corollary}
\begin{proof}
	Using Proposition \ref{treeproducts} we deduce that $L_{T'} \leq G_T$ and $K_{T'} \leq H_T$. Using Proposition \ref{treeofgroupsinjective} we deduce $H_T \leq G_T$ and $K_{T'} \leq L_{T'}$. Using Proposition \ref{treeproducts} again, we can contract the tree $T'$ to a vertex. Then $L_{T'}$ is a vertex group containing $K_{T'}$. Let $e \in ET$ be an edge joining $T'$ with a vertex of $VT \backslash VT'$ and suppose $o(e) \in T'$. As $\alpha_e(G_e) \leq G_{o(e)}$, the previous proposition yields $\alpha_e(G_e) \cap K_{T'} \leq G_{o(e)} \cap K_{T'} = H_{o(e)}$. This implies $\alpha^{-1}_e(K_{T'}) \leq \alpha^{-1}_e(H_{o(e)})$. As $H_e \leq \alpha^{-1}_e(K_{T'}) \leq \alpha^{-1}_e(H_{o(e)}) = H_e$, we deduce $\alpha^{-1}_e(K_{T'}) = H_e = \alpha^{-1}_e(H_{o(e)})$. We denote the tree products of the trees of groups $\mathbb{G}$ and $\mathbb{H}$, where $T'$ is contracted to a vertex, by $G'$ and $H'$. Using Proposition \ref{treeofgroupsinjective} the canonical homomorphism $\nu': H' \to G'$ is injective and we have $\nu'(H') \cap L_{T'} = K_{T'}$ (note that $L_{T'}$ is a vertex group of $G'$). This finishes the claim.
\end{proof}

\begin{corollary}\label{AcapBisC}
	Let $A, B, C$ be groups and let $C \to A$, $C \to B$ be two monomorphisms. Then $A \cap B = C$ in $A\star_C B$.
\end{corollary}
\begin{proof}
	Using Proposition \ref{treeofgroupsinjective} we have a monomorphism $A \cong A \star_C C \to A \star_C B$ and $A \cap B = C$.
\end{proof}

\begin{remark}\label{Remark: isomorphism preserves amalgamated product}
	Let $A', A, B, C$ be groups, let $\alpha: C \to A$, $\beta:C \to B$ and $\alpha': C \to A'$ be monomorphisms and let $\phi: A \to A'$ be an isomorphism. If $\alpha' = \phi \circ \alpha$, then the amalgamated products $A \star_C B$ and $A' \star_C B$ are isomorphic. One can prove this by constructing two unique homomorphisms $A \star_C B \to A' \star_C B$ and $A' \star_C B \to A \star_C B$ such that the concatenation is the identity on $A$ (resp.\ $A'$) and on $B$.
\end{remark}

\begin{lemma}\label{folding}
	Let $\mathbb{G} = (T, (G_v)_{v\in VT}, (G_e)_{e\in ET})$ be a tree of groups. Let $e \in ET$ and $G_e \leq H_{o(e)} \leq G_{o(e)}$. Let $VT' = VT \cup \{x\}$, $ET' = \left( ET \backslash \{ e, e^{-1} \} \right) \cup \{ f, f^{-1}, h, h^{-1} \}$ with $o(f) = o(e)$, $t(f) = x = o(h)$, $t(h) = t(e)$, $G_x := H_{o(e)} =: G_f$, $G_h := G_e$. Then the two tree products of the trees of groups are isomorphic.
\end{lemma}
\begin{proof}
	Using Proposition \ref{treeproducts}, we contract the edge $f$ to a vertex. The claim follows now from Remark \ref{Remark: isomorphism preserves amalgamated product} and the fact that $G_{o(e)} \star_{H_{o(e)}} H_{o(e)} \cong G_{o(e)}$.
\end{proof}

\subsection*{Commutator blueprints of type $\mathbf{(4, 4, 4)}$}\label{Chapter: commutator blueprint of type ((W, S), D)}

In \cite{BiRGD} we have introduced \emph{commutator blueprints} of type $(W, S)$. In this paper we are only interested in the case where $(W, S)$ is of type $(4, 4, 4)$. For more information about general commutator blueprints we refer to \cite[Section $3$]{BiRGD}.

\begin{convention}
	In this subsection we let $(W, S)$ be of type $(4, 4, 4)$.
\end{convention}

For $w\in W$ we define $\Phi(w) := \{ \alpha \in \Phi_+ \mid w \notin \alpha \}$. Let $G = (c_0, \ldots, c_k) \in \mathrm{Min}$ and let $(\alpha_1, \ldots, \alpha_k)$ be the sequence of roots crossed by $G$. We define $\Phi(G) := \{ \alpha_i \mid 1 \leq i \leq k \}$. Using the indices we obtain an order $\leq_G$ on $\Phi(G)$ and, in particular, on $[\alpha, \beta] = [\beta, \alpha] \subseteq \Phi(G)$ for all $\alpha, \beta \in \Phi(G)$. Note that $\Phi(G) = \Phi(w)$ holds for every $G \in \mathrm{Min}(w)$. We abbreviate $\mathcal{I} := \{ (G, \alpha, \beta) \in \mathrm{Min} \times \Phi_+ \times \Phi_+ \mid \alpha, \beta \in \Phi(G), \alpha \leq_G \beta \}$.

Let $\left(M_{\alpha, \beta}^G \right)_{(G, \alpha, \beta) \in \mathcal{I}}$ be a family consisting of subsets $M_{\alpha, \beta}^G \subseteq (\alpha, \beta)$ ordered via $\leq_G$. For $w\in W$ we define the group $U_w$ via the following presentation:
\[ U_w := \left\langle \{ u_{\alpha} \mid \alpha \in \Phi(w) \} \;\middle|\; \begin{cases*}
	\forall \alpha \in \Phi(w): u_{\alpha}^2 = 1, \\
	\forall (G, \alpha, \beta) \in \mathcal{I}, G \in \mathrm{Min}(w): [u_{\alpha}, u_{\beta}] = \prod\nolimits_{\gamma \in M_{\alpha, \beta}^G} u_{\gamma}
\end{cases*} \right\rangle
\]
Here the product is understood to be ordered via the order $\leq_G$, i.e.\ if $(G, \alpha, \beta) \in \mathcal{I}$ with $G \in \mathrm{Min}(w)$ and $M_{\alpha, \beta}^G = \{ \gamma_1 \leq_G \ldots \leq_G \gamma_k \} \subseteq (\alpha, \beta) \subseteq \Phi(G)$, then $\prod\nolimits_{\gamma \in M_{\alpha, \beta}^G} u_{\gamma} = u_{\gamma_1} \cdots u_{\gamma_k}$. Note that there could be $G, H \in \mathrm{Min}(w)$, $\alpha, \beta \in \Phi(w)$ with $\alpha \leq_G \beta$ and $\beta \leq_H \alpha$. In this case we have two commutation relations, namely
\begin{align*}
	&[u_{\alpha}, u_{\beta}] = \prod_{\gamma \in M_{\alpha, \beta}^G} u_{\gamma} &&\text{and} &&[u_{\beta}, u_{\alpha}] = \prod_{\gamma \in M_{\beta, \alpha}^H} u_{\gamma}.
\end{align*}
From now on we will implicitly assume that each product $\prod_{\gamma \in M_{\alpha, \beta}^G} u_{\gamma}$ is ordered via the order $\leq_G$.

\begin{definition}\label{Definition: commutator blueprint}
	A \emph{commutator blueprint of type $(4, 4, 4)$} is a family $\mathcal{M} = \left(M_{\alpha, \beta}^G \right)_{(G, \alpha, \beta) \in \mathcal{I}}$ of subsets $M_{\alpha, \beta}^G \subseteq (\alpha, \beta)$ ordered via $\leq_G$ satisfying the following axioms:
	\begin{enumerate}[label=(CB\arabic*)]
		\item Let $G = (c_0, \ldots, c_k) \in \mathrm{Min}$ and let $H = (c_0, \ldots, c_m)$ for some $1 \leq m \leq k$. Then $M_{\alpha, \beta}^H = M_{\alpha, \beta}^G$ holds for all $\alpha, \beta \in \Phi(H)$ with $\alpha \leq_H \beta$.
		
		\item Let $s\neq t \in S$, let $G \in \mathrm{Min}(r_{\{s, t\}})$, let $(\alpha_1, \ldots, \alpha_4)$ be the sequence of roots crossed by $G$ and let $1 \leq i < j \leq 4$. Then we have
		\[ M_{\alpha_i, \alpha_j}^G = \begin{cases}
			(\alpha_i, \alpha_j) & \{ \alpha_i, \alpha_j \} = \{ \alpha_s, \alpha_t \} \\
			\emptyset & \{ \alpha_i, \alpha_j \} \neq \{ \alpha_s, \alpha_t \}
		\end{cases} = \begin{cases}
		\{ \alpha_2, \alpha_3 \} & (i, j) = (1, 4), \\
		\emptyset & \text{else}.
	\end{cases} 
		 \]
		
		\item For each $w\in W$ we have $\vert U_w \vert = 2^{\ell(w)}$, where $U_w$ is defined as above.
	\end{enumerate}
	
	A commutator blueprint $\mathcal{M} = \left(M_{\alpha, \beta}^G \right)_{(G, \alpha, \beta) \in \mathcal{I}}$ of type $(4, 4, 4)$ is called \textit{locally Weyl-invariant} if for all $w\in W$, $s\in S$, $G \in \mathrm{Min}_s(w)$ and $\alpha, \beta \in \Phi(G) \backslash \{ \alpha_s \}$ with $\alpha \leq_G \beta$ and $o(r_{\alpha} r_{\beta}) < \infty$ we have $M_{s\alpha, s\beta}^{sG} = sM_{\alpha, \beta}^G := \{ s\gamma \mid \gamma \in M_{\alpha, \beta}^G \}$.
\end{definition}

\begin{remark}
	Let $\mathcal{M} = \left(M_{\alpha, \beta}^G \right)_{(G, \alpha, \beta) \in \mathcal{I}}$ be a commutator blueprint of type $(4, 4, 4)$. Then $\mathcal{M}$ is locally Weyl-invariant if and only if for all $(G, \alpha, \beta) \in \mathcal{I}$ with $o(r_{\alpha} r_{\beta}) < \infty$ the following holds:
	\[ M_{\alpha, \beta}^G = \begin{cases}
		(\alpha, \beta) & \text{ if }\vert (\alpha, \beta) \vert = 2, \\
		\emptyset & \text{ else}.
	\end{cases} \]
\end{remark}

\begin{remark}\label{Remark: Product mapping is a bijection}
	Let $G = (c_0, \ldots, c_k) \in \mathrm{Min}(w)$ and let $(\alpha_1, \ldots, \alpha_k)$ be the sequence of roots crossed by $G$. Note that it is a direct consequence of (CB$3$) that the product map $U_{\alpha_1} \times \cdots \times U_{\alpha_k} \to U_w, (u_1, \ldots, u_k) \mapsto u_1 \cdots u_k$ is a bijection, where $\ZZ_2 \cong U_{\alpha_i} = \langle u_{\alpha_i} \rangle \leq U_w$.
\end{remark}

\subsection*{Tree products}

We have mentioned in the introduction that this paper can be seen as the first in a series of two. To ease citation, we use from now on the same notation for both.

\begin{convention}
	For the rest of this paper we $(W, S)$ be of type $(4, 4, 4)$ and $S = \{ r, s, t \}$. Moreover, we let $\mathcal{M} = \left( M_{\alpha, \beta}^G \right)_{(G, \alpha, \beta) \in \mathcal{I}}$ be a locally Weyl-invariant commutator blueprint of type $(4, 4, 4)$.
\end{convention}

For a residue $R$ of $\Sigma(W, S)$ we put $w_R := \proj_R 1_W$. Let $R$ be a residue of type $\{ s, t \}$. Then we have $\ell(w_R s) = \ell(w_R) +1 = \ell(w_R t)$. We define the group $V_{w_R r_{\{ s, t \}}} := \langle U_{w_R s}, U_{w_R t} \rangle \leq U_{w_R r_{\{s, t\}}}$. Using (CB$3$) and fact that $\mathcal{M}$ is locally Weyl-invariant, the group $V_{w_R r_{\{s, t\}}}$ is an index $2$ subgroup of $U_{w_R r_{\{ s, t \}}}$ (cf.\ Remark \ref{Remark: Product mapping is a bijection}). For each $i \in \NN$ we let $\mathcal{R}_i$ be the set of all rank $2$ residues $R$ with $\ell(w_R) = i$ (e.g.\ $\mathcal{R}_0 = \{ R_{\{s, t\}}(1_W) \mid s\neq t \in S \}$). We let $\mathcal{T}_{i, 1}$ be the set of all residues $R \in \mathcal{R}_i$ with $\ell(w_R sr) = \ell(w_R) +2 = \ell(w_R tr)$, where $\{ s, t \}$ is the type of $R$.

The main goal of the rest of this paper is to prove Theorem \ref{Theorem: Main application}. Therefore, we need to introduce several sequences of groups. The groups in the sequences of groups will always be generated by elements $u_{\alpha}$ for suitable $\alpha \in \Phi_+$. Let $A$ and $B$ vertex groups such that the corresponding vertices are joint by an edge, and let $C$ be the edge group. Let $\Phi_A, \Phi_B \subseteq \Phi_+$ be such that $A = \langle u_{\alpha} \mid \alpha \in \Phi_A \rangle$ and $B = \langle u_{\alpha} \mid \alpha \in \Phi_B \rangle$. If we do not specify $C$, then we will implicitly assume that $C = \langle u_{\alpha} \mid \alpha \in \Phi_A \cap \Phi_B \rangle$. If $C$ is as in this case, then it will always be clear that we have canonical homomorphisms $C \to A$ and $C \to B$ which are injective, and we define $A \hat{\star} B := A \star_C B$.

\subsection*{The groups $\mathbf{V_R}$ and $\mathbf{O_R}$}

For a residue $R \in \mathcal{T}_{i, 1}$ of type $\{s, t\}$ we define the group $V_R$ to be the tree product of the sequence of groups with vertex groups
\allowdisplaybreaks
\begin{align*}
	U_{w_R sr}, V_{w_R r_{\{s, t\}}}, U_{w_R tr}
\end{align*}
Moreover, we define the group $O_R$ to be the tree product of the sequence of groups with vertex groups
\allowdisplaybreaks
\begin{align*}
	V_{w_R sr_{\{r, t\}}}, U_{w_R r_{\{s, t\}}}, V_{w_R tr_{\{r, s\}}}
\end{align*}

\begin{remark}\label{Remark: tree product generated by root group elements}
	For $V_R$ we consider $\alpha := w_R s\alpha_r$. Using Lemma \ref{mingallinrep} we see that $-w_R \alpha_t \subseteq \alpha$. As $w_R t \in (-w_R \alpha_t)$, we deduce $w_R tr, w_R r_{\{s, t\}} \in \alpha$ and hence $u_{\alpha}$ is neither a generator of $V_{w_R r_{\{s, t\}}}$ nor of $U_{w_R tr}$. Now we consider $w_R \alpha_s$. As $-w_R t\alpha_r \subseteq w_R \alpha_s$ by Lemma \ref{mingallinrep} we deduce that $u_{w_R \alpha_s}$ is not a generator of $U_{w_R tr}$. Using similar methods we infer that $V_R$ is generated by $\{ u_{\alpha} \mid \exists v\in \{ w_R sr, w_R tr  \}: v\notin \alpha \}$. A similar result holds for $O_R$.
\end{remark}

\begin{figure}[h]
	\begin{minipage}{0.4\linewidth}
		\begin{tikzpicture}[transform shape, scale=0.8]
			\node[node] (a) at (0,0) {};
			\node[node] (b) at (0,2) {};
			\node[node] (c) at (-1,3) {};
			\node[node] (d) at (1,3) {};
			\node[node] (e) at (-2.5,4) {};
			\node[node, gray] (f) at (-1,4) {};
			\node[node, gray] (g) at (1,4) {};
			\node[node] (h) at (2.5,4) {};
			\node[node, gray] (i) at (-1,5) {};
			\node[node, gray] (j) at (1,5) {};
			\node[node, gray] (k) at (0,6) {};
			
			\node[right] () at (0.1,0) {$1_W$};
			\node[below right] () at (0.1,2) {$w_R$};
			\node () at (0,4) {$R$};
			
			\draw[-, very thick] (a) to (b);
			\draw[-, very thick] (b) to node[above] {$s$} (c);
			\draw[-, very thick] (b) to node[above] {$t$} (d);
			\draw[-, very thick] (c) to node[above] {$r$} (e);
			\draw[-, gray] (c) to node[right] {$t$} (f);
			\draw[-, gray] (d) to node[left] {$s$} (g);
			\draw[-, very thick] (d) to node[above] {$r$} (h);
			\draw[-, gray] (f) to node[right] {$s$} (i);
			\draw[-, gray] (g) to node[left] {$t$} (j);
			\draw[-, gray] (i) to node[below] {$t$} (k);
			\draw[-, gray] (j) to node[below] {$s$} (k);
		\end{tikzpicture}
		\caption{Illustration of the group $V_R$}
	\end{minipage}
	\begin{minipage}{0.4\linewidth}
		\begin{tikzpicture}[transform shape, scale=0.8]
			\node[node] (a) at (0,0) {};
			\node[node] (b) at (0,2) {};
			\node[node] (c) at (-1,3) {};
			\node[node] (d) at (1,3) {};
			\node[node] (e) at (-2.5,4) {};
			\node[node] (f) at (-1,4) {};
			\node[node] (g) at (1,4) {};
			\node[node] (h) at (2.5,4) {};
			\node[node] (i) at (-1,5) {};
			\node[node] (j) at (1,5) {};
			\node[node] (k) at (0,6) {};
			
			\node[right] () at (0.1,0) {$1_W$};
			\node[below right] () at (0.1,2) {$w_R$};
			\node () at (0,4) {$R$};
			
			\draw[-, very thick] (a) to (b);
			\draw[-, very thick] (b) to node[above] {$s$} (c);
			\draw[-, very thick] (b) to node[above] {$t$} (d);
			\draw[-, very thick] (c) to node[above] {$r$} (e);
			\draw[-, very thick] (c) to node[right] {$t$} (f);
			\draw[-, very thick] (d) to node[left] {$s$} (g);
			\draw[-, very thick] (d) to node[above] {$r$} (h);
			\draw[-, very thick] (f) to node[right] {$s$} (i);
			\draw[-, very thick] (g) to node[left] {$t$} (j);
			\draw[-, very thick] (i) to node[below] {$t$} (k);
			\draw[-, very thick] (j) to node[below] {$s$} (k);
		\end{tikzpicture}
		\caption{Illustration of the group $O_R$}
	\end{minipage}
\end{figure}

\begin{lemma}\label{VRtoORinjective}
	Let $R \in \mathcal{T}_{i, 1}$. Then the canonical homomorphism $V_R \to O_R$ is injective.
\end{lemma}
\begin{proof}
	Let $R$ be of type $\{s, t\}$. We will apply Proposition \ref{treeofgroupsinjective}. Therefore we first see that each vertex group of $V_R$ is contained in the corresponding vertex group of $O_R$, e.g.\ $U_{w_R sr} \leq V_{w_R sr_{\{r, t\}}}$. Next we have to show that the preimages of the boundary monomorphisms are equal and coincide with the edge groups of $V_R$. For this we compute $H_{o(e)} \cap \alpha_e(G_e)$ and $H_{t(e)} \cap \omega_e(G_e)$, as $\alpha_e^{-1}(H_{o(e)}) = \alpha_e^{-1}( H_{o(e)} \cap \alpha_e(G_e) )$ and $\omega_e^{-1}( H_{t(e)} ) = \omega_e^{-1}( H_{t(e)} \cap \omega_e(G_e) )$. We compute the following:
	\allowdisplaybreaks
	\begin{align*}
		U_{w_R sr} \cap U_{w_R st} = U_{w_R s} = V_{w_R r_{\{s, t\}}} \cap U_{w_R st} \\
		V_{w_R r_{\{s, t\}}} \cap U_{w_R ts} = U_{w_R t} = U_{w_R tr} \cap U_{w_R ts}
	\end{align*}
	Now the claim follows from Proposition \ref{treeofgroupsinjective}.
\end{proof}

\subsection*{The groups $\mathbf{V_{R, s}}$ and $\mathbf{O_{R, s}}$}

Let $R \in \mathcal{T}_{i, 1}$ be a residue of type $\{s, t\}$ such that $\ell(w_R srs) = \ell(w_R) +3$. Then we define the group $V_{R, s}$ to be the tree product of the sequence of groups with vertex groups
\allowdisplaybreaks
\begin{align*}
	U_{w_R srs}, V_{w_R r_{\{s, t\}}}, U_{w_R tr}
\end{align*}
Moreover, we define the group $O_{R, s}$ to be the tree product of the sequence of groups with vertex groups
\allowdisplaybreaks
\begin{align*}
	U_{w_R srs}, V_{w_R sr_{\{r, t\}}}, U_{w_R r_{\{s, t\}}}, V_{w_R t r_{\{r, s\}}}
\end{align*}
It follows similarly as in Remark \ref{Remark: tree product generated by root group elements} that $V_{R, s}$ and $O_{R, s}$ are generated by suitable $u_{\alpha}$.

\begin{figure}[h]
	\begin{minipage}{0.4\linewidth}
		\begin{tikzpicture}[transform shape, scale=0.8]
			\node[node] (a) at (0,0) {};
			\node[node] (b) at (0,2) {};
			\node[node] (c) at (-1,3) {};
			\node[node] (d) at (1,3) {};
			\node[node] (e) at (-2.5,4) {};
			\node[node, gray] (f) at (-1,4) {};
			\node[node, gray] (g) at (1,4) {};
			\node[node] (h) at (2.5,4) {};
			\node[node, gray] (i) at (-1,5) {};
			\node[node, gray] (j) at (1,5) {};
			\node[node, gray] (k) at (0,6) {};
			\node[node] (l) at (-3.5, 5) {};
			
			\node[right] () at (0.1,0) {$1_W$};
			\node[below right] () at (0.1,2) {$w_R$};
			\node () at (0,4) {$R$};
			
			\draw[-, very thick] (a) to (b);
			\draw[-, very thick] (b) to node[above] {$s$} (c);
			\draw[-, very thick] (b) to node[above] {$t$} (d);
			\draw[-, very thick] (c) to node[above] {$r$} (e);
			\draw[-, gray] (c) to node[right] {$t$} (f);
			\draw[-, gray] (d) to node[left] {$s$} (g);
			\draw[-, very thick] (d) to node[above] {$r$} (h);
			\draw[-, gray] (f) to node[right] {$s$} (i);
			\draw[-, gray] (g) to node[left] {$t$} (j);
			\draw[-, gray] (i) to node[below] {$t$} (k);
			\draw[-, gray] (j) to node[below] {$s$} (k);
			\draw[-, very thick] (e) to node[above] {$s$} (l);
		\end{tikzpicture}
		\caption{Illustration of the group $V_{R, s}$}
	\end{minipage}
	\begin{minipage}{0.4\linewidth}
		\begin{tikzpicture}[transform shape, scale=0.8]
			\node[node] (a) at (0,0) {};
			\node[node] (b) at (0,2) {};
			\node[node] (c) at (-1,3) {};
			\node[node] (d) at (1,3) {};
			\node[node] (e) at (-2.5,4) {};
			\node[node] (f) at (-1,4) {};
			\node[node] (g) at (1,4) {};
			\node[node] (h) at (2.5,4) {};
			\node[node] (i) at (-1,5) {};
			\node[node] (j) at (1,5) {};
			\node[node] (k) at (0,6) {};
			\node[node] (l) at (-3.5, 5) {};
			
			\node[right] () at (0.1,0) {$1_W$};
			\node[below right] () at (0.1,2) {$w_R$};
			\node () at (0,4) {$R$};
			
			\draw[-, very thick] (a) to (b);
			\draw[-, very thick] (b) to node[above] {$s$} (c);
			\draw[-, very thick] (b) to node[above] {$t$} (d);
			\draw[-, very thick] (c) to node[above] {$r$} (e);
			\draw[-, very thick] (c) to node[right] {$t$} (f);
			\draw[-, very thick] (d) to node[left] {$s$} (g);
			\draw[-, very thick] (d) to node[above] {$r$} (h);
			\draw[-, very thick] (f) to node[right] {$s$} (i);
			\draw[-, very thick] (g) to node[left] {$t$} (j);
			\draw[-, very thick] (i) to node[below] {$t$} (k);
			\draw[-, very thick] (j) to node[below] {$s$} (k);
			\draw[-, very thick] (e) to node[above] {$s$} (l);
		\end{tikzpicture}
		\caption{Illustration of the group $O_{R, s}$}
	\end{minipage}
\end{figure}

\begin{lemma}\label{Lemma: V_R,s to O_R,s injective}
	Let $R \in \mathcal{T}_{i, 1}$ be a residue of type $\{s, t\}$ such that $\ell(w_R srs) = \ell(w_R) +3$. Then the canonical homomorphisms $V_R \to V_{R, s}, O_R \to O_{R, s}$ and $V_{R, s} \to O_{R, s}$ are injective. Moreover, we have $V_{R, s} \star_{V_R} O_R \cong U_{w_R srs} \star_{U_{w_R sr}} O_R \cong O_{R, s}$.
\end{lemma}
\begin{proof}
	Note that $V_{R, s} \cong U_{w_R srs} \star_{U_{w_R sr}} V_R$ and $O_{R, s} \cong U_{w_R srs} \star_{U_{w_R sr}} O_R$ by Proposition \ref{treeproducts} and Lemma \ref{folding}. In particular, Proposition \ref{treeproducts} implies that $V_R \to V_{R, s}$ and $O_R \to O_{R, s}$ are injective. Using Proposition \ref{treeofgroupsinjective} and Lemma \ref{VRtoORinjective}, the canonical homomorphism $V_{R, s} \to O_{R, s}$ is injective. We deduce from Proposition \ref{treeproducts}, Remark \ref{Remark: isomorphism preserves amalgamated product} and Lemma \ref{folding} the following:
	\[ O_{R, s}  \cong U_{w_R srs} \star_{U_{w_R sr}} O_R \cong (U_{w_R srs} \star_{U_{w_R sr}} V_R) \star_{V_R} O_R \cong V_{R, s} \star_{V_R} O_R \qedhere \]
\end{proof}

\subsection*{The groups $\mathbf{H_R}$ and $\mathbf{K_{R, s}}$}\label{Subsection: H_R}

Let $R \in \mathcal{T}_{i, 1}$ be of type $\{s, t\}$. We define the group $H_R$ to be the tree product of the sequence of groups with vertex groups
\allowdisplaybreaks
\begin{align*}
	U_{w_Rsr_{\{r, t\}}}, V_{w_R str_{\{r, s\}}}, U_{w_R r_{\{s, t\}}}, V_{w_R ts r_{\{r, t\}}}, U_{w_R t r_{\{r, s\}}}
\end{align*}
Moreover, we define the group $K_{R, s}$ to be the tree product of the sequence of groups with vertex groups
\allowdisplaybreaks
\begin{align*}
	U_{w_R sr_{\{r, t\}}}, V_{w_R str_{\{r, s\}}}, U_{w_R r_{\{s, t\}}}, V_{w_R tr_{\{r, s\}}}
\end{align*}
It follows similarly as in Remark \ref{Remark: tree product generated by root group elements} that $H_R$ and $K_{R, s}$ are generated by suitable $u_{\alpha}$.

\begin{figure}[h]
	\centering
	\begin{tikzpicture}[transform shape, scale=0.7]
		\node[node] (a) at (0,0) {};
		\node[node] (b) at (0,2) {};
		\node[node] (c) at (-1,3) {};
		\node[node] (d) at (1,3) {};
		\node[node] (e) at (-2.5,4) {};
		\node[node] (f) at (-1,4) {};
		\node[node] (g) at (1,4) {};
		\node[node] (h) at (2.5,4) {};
		\node[node] (i) at (-4.5,5) {};
		\node[node] (j) at (-2,5) {};
		\node[node] (k) at (-1,5) {};
		\node[node] (l) at (1,5) {};
		\node[node] (m) at (2,5) {};
		\node[node] (n) at (4.5,5) {};
		\node[node] (o) at (-6,6) {};
		\node[node] (p) at (-4.5,6) {};
		\node[node, gray] (q) at (-3,6) {};
		\node[node, gray] (r) at (-1.5,6) {};
		\node[node] (s) at (0,6) {};
		\node[node, gray] (t) at (1.5,6) {};
		\node[node, gray] (u) at (3,6) {};
		\node[node] (v) at (4.5,6) {};
		\node[node] (w) at (6,6) {};
		\node[node] (x) at (-6,7) {};
		\node[node, gray] (y) at (-4.5,7) {};
		\node[node, gray] (z) at (-3,7) {};
		\node[node, gray] (aa) at (-2,7) {};
		\node[node, gray] (ab) at (0,7) {};
		\node[node, gray] (ac) at (2,7) {};
		\node[node, gray] (ad) at (3,7) {};
		\node[node, gray] (ae) at (4.5,7) {};
		\node[node] (af) at (6,7) {};
		\node[node, gray] (ag) at (-4,8) {};
		\node[node, gray] (ah) at (-2.5,8) {};
		\node[node, gray] (ai) at (-1,8) {};
		\node[node, gray] (aj) at (1,8) {};
		\node[node, gray] (ak) at (2.5,8) {};
		\node[node, gray] (al) at (4,8) {};
		\node[node, gray] (am) at (-2,9) {};
		\node[node, gray] (an) at (2,9) {};
		
		\node[right] () at (0.1,0) {$1_W$};
		\node[below right] () at (0.1,2) {$w_R$};
		\node () at (0,4) {$R$};
		
		\draw[-, very thick] (a) to (b);
		\draw[-, very thick] (b) to node[above] {$s$} (c);
		\draw[-, very thick] (b) to node[above] {$t$} (d);
		\draw[-, very thick] (c) to node[above] {$r$} (e);
		\draw[-, very thick] (e) to node[above] {$t$} (i);
		\draw[-, very thick] (i) to node[above] {$r$} (o);
		\draw[-, very thick] (o) to node[left] {$t$} (x);
		\draw[-, very thick] (c) to node[right] {$t$} (f);
		\draw[-, very thick] (f) to node[right] {$s$} (k);
		\draw[-, very thick] (k) to node[below] {$t$} (s);
		\draw[-, very thick] (d) to node[left] {$s$} (g);
		\draw[-, very thick] (g) to node[left] {$t$} (l);
		\draw[-, very thick] (l) to node[below] {$s$} (s);
		\draw[-, very thick] (f) to node[below] {$r$} (j);
		\draw[-, very thick] (j) to node[below] {$t$} (p);
		\draw[-, very thick] (p) to node[below] {$r$} (x);
		\draw[-, gray] (j) to node[above] {$s$} (q);
		\draw[-, gray] (q) to node[above] {$r$} (y);
		\draw[-, gray] (y) to node[left] {$s$} (ag);
		\draw[-, gray] (k) to node[right] {$r$} (r);
		\draw[-, gray] (r) to node[below] {$s$} (z);
		\draw[-, gray] (z) to node[right] {$r$} (ag);
		\draw[-, gray] (r) to node[right] {$t$} (aa);
		\draw[-, gray] (aa) to node[right] {$r$} (ah);
		\draw[-, gray] (ah) to node[left] {$t$} (am);
		\draw[-, gray] (s) to node[left] {$r$} (ab);
		\draw[-, gray] (ab) to node[left] {$t$} (ai);
		\draw[-, gray] (ai) to node[right] {$r$} (am);
		\draw[-, very thick] (d) to node[above] {$r$} (h);
		\draw[-, very thick] (h) to node[above] {$s$} (n);
		\draw[-, very thick] (n) to node[above] {$r$} (w);
		\draw[-, very thick] (w) to node[right] {$s$} (af);
		\draw[-, very thick] (g) to node[below] {$r$} (m);
		\draw[-, very thick] (m) to node[below] {$s$} (v);
		\draw[-, very thick] (v) to node[below] {$r$} (af);
		\draw[-, gray] (m) to node[above] {$t$} (u);
		\draw[-, gray] (u) to node[above] {$r$} (ae);
		\draw[-, gray] (ae) to node[right] {$t$} (al);
		\draw[-, gray] (l) to node[left] {$r$} (t);
		\draw[-, gray] (t) to node[below] {$t$} (ad);
		\draw[-, gray] (ad) to node[left] {$r$} (al);
		\draw[-, gray] (t) to node[left] {$s$} (ac);
		\draw[-, gray] (ac) to node[left] {$r$} (ak);
		\draw[-, gray] (ak) to node[right] {$s$} (an);
		\draw[-, gray] (ab) to node[right] {$s$} (aj);
		\draw[-, gray] (aj) to node[left] {$r$} (an);
	\end{tikzpicture}
	\caption{Illustration of the group $H_R$}
\end{figure}
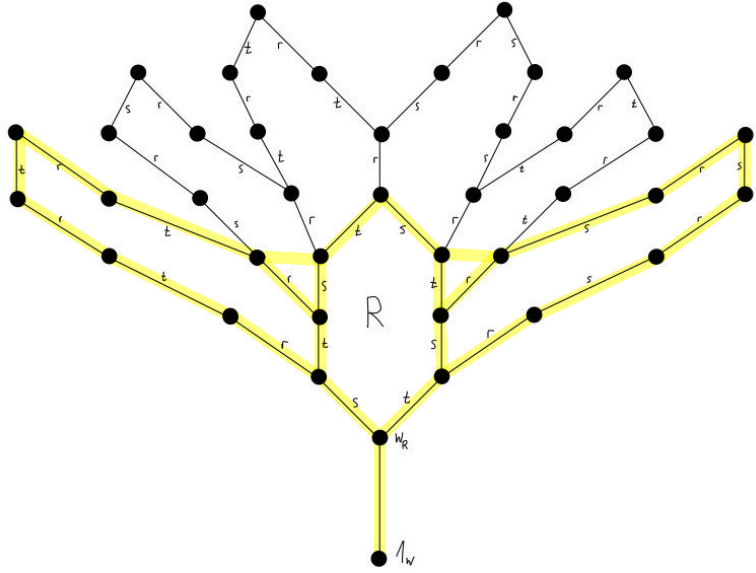

\begin{figure}[h]
	\centering
	\begin{tikzpicture}[transform shape, scale=0.7]
		\node[node] (a) at (0,0) {};
		\node[node] (b) at (0,2) {};
		\node[node] (c) at (-1,3) {};
		\node[node] (d) at (1,3) {};
		\node[node] (e) at (-2.5,4) {};
		\node[node] (f) at (-1,4) {};
		\node[node] (g) at (1,4) {};
		\node[node] (h) at (2.5,4) {};
		\node[node] (i) at (-4.5,5) {};
		\node[node] (j) at (-2,5) {};
		\node[node] (k) at (-1,5) {};
		\node[node] (l) at (1,5) {};
		\node[node, gray] (m) at (2,5) {};
		\node[node, gray] (n) at (4.5,5) {};
		\node[node] (o) at (-6,6) {};
		\node[node] (p) at (-4.5,6) {};
		\node[node, gray] (q) at (-3,6) {};
		\node[node, gray] (r) at (-1.5,6) {};
		\node[node] (s) at (0,6) {};
		\node[node, gray] (t) at (1.5,6) {};
		\node[node, gray] (u) at (3,6) {};
		\node[node, gray] (v) at (4.5,6) {};
		\node[node, gray] (w) at (6,6) {};
		\node[node] (x) at (-6,7) {};
		\node[node, gray] (y) at (-4.5,7) {};
		\node[node, gray] (z) at (-3,7) {};
		\node[node, gray] (aa) at (-2,7) {};
		\node[node, gray] (ab) at (0,7) {};
		\node[node, gray] (ac) at (2,7) {};
		\node[node, gray] (ad) at (3,7) {};
		\node[node, gray] (ae) at (4.5,7) {};
		\node[node, gray] (af) at (6,7) {};
		\node[node, gray] (ag) at (-4,8) {};
		\node[node, gray] (ah) at (-2.5,8) {};
		\node[node, gray] (ai) at (-1,8) {};
		\node[node, gray] (aj) at (1,8) {};
		\node[node, gray] (ak) at (2.5,8) {};
		\node[node, gray] (al) at (4,8) {};
		\node[node, gray] (am) at (-2,9) {};
		\node[node, gray] (an) at (2,9) {};
		
		\node[right] () at (0.1,0) {$1_W$};
		\node[below right] () at (0.1,2) {$w_R$};
		\node () at (0,4) {$R$};
		
		\draw[-, very thick] (a) to (b);
		\draw[-, very thick] (b) to node[above] {$s$} (c);
		\draw[-, very thick] (b) to node[above] {$t$} (d);
		\draw[-, very thick] (c) to node[above] {$r$} (e);
		\draw[-, very thick] (e) to node[above] {$t$} (i);
		\draw[-, very thick] (i) to node[above] {$r$} (o);
		\draw[-, very thick] (o) to node[left] {$t$} (x);
		\draw[-, very thick] (c) to node[right] {$t$} (f);
		\draw[-, very thick] (f) to node[right] {$s$} (k);
		\draw[-, very thick] (k) to node[below] {$t$} (s);
		\draw[-, very thick] (d) to node[left] {$s$} (g);
		\draw[-, very thick] (g) to node[left] {$t$} (l);
		\draw[-, very thick] (l) to node[below] {$s$} (s);
		\draw[-, very thick] (f) to node[below] {$r$} (j);
		\draw[-, very thick] (j) to node[below] {$t$} (p);
		\draw[-, very thick] (p) to node[below] {$r$} (x);
		\draw[-, gray] (j) to node[above] {$s$} (q);
		\draw[-, gray] (q) to node[above] {$r$} (y);
		\draw[-, gray] (y) to node[left] {$s$} (ag);
		\draw[-, gray] (k) to node[right] {$r$} (r);
		\draw[-, gray] (r) to node[below] {$s$} (z);
		\draw[-, gray] (z) to node[right] {$r$} (ag);
		\draw[-, gray] (r) to node[right] {$t$} (aa);
		\draw[-, gray] (aa) to node[right] {$r$} (ah);
		\draw[-, gray] (ah) to node[left] {$t$} (am);
		\draw[-, gray] (s) to node[left] {$r$} (ab);
		\draw[-, gray] (ab) to node[left] {$t$} (ai);
		\draw[-, gray] (ai) to node[right] {$r$} (am);
		\draw[-, very thick] (d) to node[above] {$r$} (h);
		\draw[-, gray] (h) to node[above] {$s$} (n);
		\draw[-, gray] (n) to node[above] {$r$} (w);
		\draw[-, gray] (w) to node[right] {$s$} (af);
		\draw[-, gray] (g) to node[below] {$r$} (m);
		\draw[-, gray] (m) to node[below] {$s$} (v);
		\draw[-, gray] (v) to node[below] {$r$} (af);
		\draw[-, gray] (m) to node[above] {$t$} (u);
		\draw[-, gray] (u) to node[above] {$r$} (ae);
		\draw[-, gray] (ae) to node[right] {$t$} (al);
		\draw[-, gray] (l) to node[left] {$r$} (t);
		\draw[-, gray] (t) to node[below] {$t$} (ad);
		\draw[-, gray] (ad) to node[left] {$r$} (al);
		\draw[-, gray] (t) to node[left] {$s$} (ac);
		\draw[-, gray] (ac) to node[left] {$r$} (ak);
		\draw[-, gray] (ak) to node[right] {$s$} (an);
		\draw[-, gray] (ab) to node[right] {$s$} (aj);
		\draw[-, gray] (aj) to node[left] {$r$} (an);
	\end{tikzpicture}
	\caption{Illustration of the group $K_{R, s}$}
\end{figure}
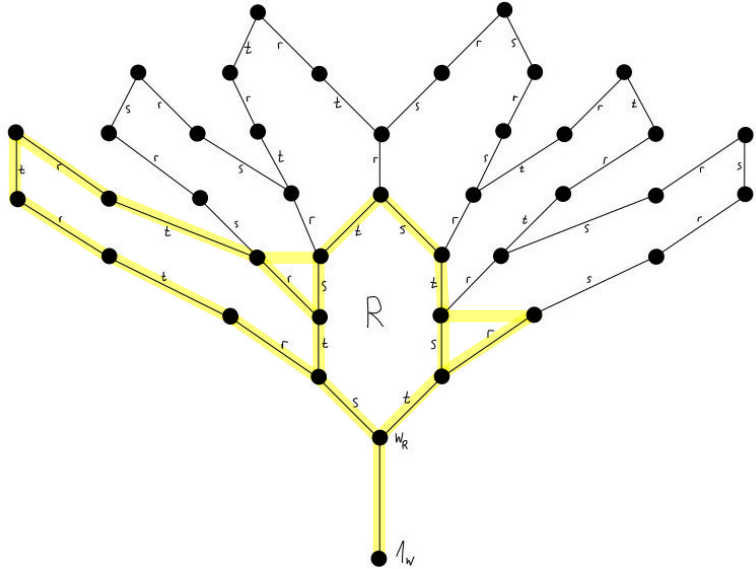

\begin{lemma}\label{CCleftCright}
	Let $R \in \mathcal{T}_{i, 1}$ be of type $\{s, t\}$. Then the canonical homomorphisms $O_R \to K_{R, s}, K_{R, t}$ are injective and we have $H_R \cong K_{R, s} \star_{O_R} K_{R, t}$.
\end{lemma}
\begin{proof}
	Using Proposition \ref{treeproducts} the group $K_{R, s}$ is isomorphic to the tree product of the sequence of groups with vertex groups
	\allowdisplaybreaks
	\begin{align*}
		U_{w_R sr_{\{r, t\}}}, V_{w_R str_{\{r, s\}}} \hat{\star} U_{w_R r_{\{s, t\}}}, V_{w_R tr_{\{r, s\}}}
	\end{align*}
	One easily sees that each vertex group of $O_R$ is contained in the corresponding vertex group of the previous tree product. Considering the preimage of the boundary monomorphisms the following holds (using Corollary \ref{AcapBisC}):
	\allowdisplaybreaks
	\begin{align*}
		V_{w_R sr_{\{r, t\}}} \cap U_{w_R str} = U_{w_R st} = U_{w_R r_{\{s, t\}}} \cap U_{w_R str}
	\end{align*}
	As before, Proposition \ref{treeofgroupsinjective} yields that the canonical homomorphism $O_R \to K_{R, s}$ is injective. Using similar arguments, we obtain that $O_R \to K_{R, t}$ is injective. We define $C_0 := V_{w_R sr_{\{r, t\}}} \hat{\star} U_{w_R r_{\{s, t\}}}$ and note that $U_{w_R tst} \to C_0$ and $C_0 \to O_R$ are injective. Moreover, the computations above imply that $C_0 \to U_{w_R s r_{\{r, t\}}} \hat{\star} V_{w_R str_{\{r, s\}}} \hat{\star} U_{w_R r_{\{s, t\}}}$ is injective. Now the following isomorphisms follow from Proposition \ref{treeproducts}, Remark \ref{Remark: isomorphism preserves amalgamated product} and Lemma \ref{folding}:
	\allowdisplaybreaks
	\begin{align*}
		H_R &\cong \left( U_{w_R s r_{\{r, t\}}} \hat{\star} V_{w_R st r_{\{r, s\}}} \hat{\star} U_{w_R r_{\{s, t\}}} \right) \star_{U_{w_R tst}} \left( V_{w_R ts r_{\{r, t\}}}\hat{\star} U_{w_R t r_{\{r, s\}}} \right) \\
		&\cong \left( U_{w_R s r_{\{r, t\}}} \hat{\star} V_{w_R st r_{\{r, s\}}} \hat{\star} U_{w_R r_{\{s, t\}}} \right) \star_{C_0} C_0 \star_{U_{w_R tst}} \left( V_{w_R ts r_{\{r, t\}}}\hat{\star} U_{w_R t r_{\{r, s\}}} \right) \\
		&\cong \left( U_{w_R s r_{\{r, t\}}} \hat{\star} V_{w_R st r_{\{r, s\}}} \hat{\star} U_{w_R r_{\{s, t\}}} \right) \star_{C_0} K_{R, t} \\
		&\cong \left( U_{w_R s r_{\{r, t\}}} \hat{\star} V_{w_R st r_{\{r, s\}}} \hat{\star} U_{w_R r_{\{s, t\}}} \right) \star_{C_0} O_R \star_{O_R} K_{R, t} \\
		&\cong \left( \left( U_{w_R s r_{\{r, t\}}} \hat{\star} V_{w_R st r_{\{r, s\}}} \hat{\star} U_{w_R r_{\{s, t\}}} \right) \star_{C_0} O_R \right) \star_{O_R} K_{R, t} \\
		&\cong \left( \left( U_{w_R s r_{\{r, t\}}} \hat{\star} V_{w_R st r_{\{r, s\}}} \hat{\star} U_{w_R r_{\{s, t\}}} \right) \star_{C_0} C_0 \star_{U_{w_R ts}} V_{w_R tr_{\{r, s\}}} \right) \star_{O_R} K_{R, t} \\
		&\cong \left( \left( U_{w_R s r_{\{r, t\}}} \hat{\star} V_{w_R st r_{\{r, s\}}} \hat{\star} U_{w_R r_{\{s, t\}}} \right) \star_{U_{w_R ts}} V_{w_R tr_{\{r, s\}}} \right) \star_{O_R} K_{R, t} \\
		&\cong K_{R, s} \star_{O_R} K_{R, t} \qedhere
	\end{align*}
\end{proof}

\begin{lemma}\label{Lemma: J_R,t cong H_R star V_T O_T}
	Let $R \in \mathcal{T}_{i, 1}$ be a residue of type $\{s, t\}$ and let $T = R_{\{r, t\}}(w_R ts)$. Then $T \in \mathcal{T}_{i+2, 1}$ and the canonical homomorphism $V_T \to H_R$ is injective.
\end{lemma}
\begin{proof}
	Note that $T \in \mathcal{T}_{i+2, 1}$. By Proposition \ref{treeproducts}, $U_{w_R r_{\{s, t\}}} \hat{\star} V_{w_R ts r_{\{r, t\}}} \hat{\star} U_{w_R tr_{\{r, s\}}} \to H_R$ is injective. Using Proposition \ref{treeofgroupsinjective}, we deduce that
	\allowdisplaybreaks
	\begin{align*}
		V_T = U_{w_R tsts} \hat{\star} V_{w_R tsr_{\{r, t\}}} \hat{\star} U_{w_R tsrs} \to U_{w_R r_{\{s, t\}}} \hat{\star} V_{w_R ts r_{\{r, t\}}} \hat{\star} U_{w_R tr_{\{r, s\}}}
	\end{align*}
	is injective and hence also the concatenation $V_T \to H_R$.
\end{proof}

\begin{remark}
	Let $R \in \mathcal{T}_{i, 1}$ be of type $\{s, t\}$ such that $\ell(w_R srs) = \ell(w_R) +3$ and let $T = R_{\{r, t\}}(w_R s)$. In the next lemma we consider $O_{R, s} \star_{V_T} O_T$. Similar as in Remark \ref{Remark: tree product generated by root group elements} we will show that if $x_{\alpha}$ is a generator of $O_{R, s}$ and $y_{\alpha}$ is a generator of $O_T$, then $x_{\alpha} = y_{\alpha}$ holds in $O_{R, s} \star_{V_T} O_T$. It suffices to consider $w_R \alpha_t$ and $w_R t\alpha_r$. As $-w_R \alpha_s \subseteq w_R t\alpha_r$ and $-w_Rs \alpha_r, -w_Rst \alpha_r \subseteq w_R \alpha_t$ by Lemma \ref{mingallinrep}, we deduce that $x_{\alpha}$ is not a generator of $O_T$ for $\alpha \in \{ w_R \alpha_t, w_Rt \alpha_r \}$.
\end{remark}

\begin{lemma}\label{CleftCrightisos}
	Let $R \in \mathcal{T}_{i, 1}$ be of type $\{s, t\}$ such that $\ell(w_R srs) = \ell(w_R) +3$ and let $T = R_{\{r, t\}}(w_R s)$. Then the canonical homomorphisms $V_T \to O_{R, s}$ and $K_{R, s} \to O_{R, s} \star_{V_T} O_T$ are injective and we have $K_{R, s} \cap O_{R, s} = O_R$ in $O_{R, s} \star_{V_T} O_T$.
\end{lemma}
\begin{proof}
	We have $O_{R, s} \cong V_T \star_{U_{w_R sts}} U_{w_R r_{\{s, t\}}} \hat{\star} V_{w_R t r_{\{r, s\}}}$ by Lemma \ref{folding} and Proposition \ref{treeproducts}. Now Proposition \ref{treeproducts} yields that the mapping $V_T \to O_{R, s}$ is injective. This, together with Proposition \ref{treeproducts}, Remark \ref{Remark: isomorphism preserves amalgamated product}, Lemma \ref{folding} and Lemma \ref{VRtoORinjective} yields the following isomorphisms:
	\allowdisplaybreaks
	\begin{align*}
		O_{R, s} \star_{V_T} O_T &\cong \left( V_T \star_{U_{w_R sts}} U_{w_R r_{\{s, t\}}} \hat{\star} V_{w_R t r_{\{r, s\}}} \right) \star_{V_T} O_T \\
		&\cong V_{w_R t r_{\{r, s\}}} \hat{\star} U_{w_R r_{\{s, t\}}} \star_{U_{w_R sts}} V_T \star_{V_T} O_T \\
		&\cong V_{w_R t r_{\{r, s\}}} \hat{\star} U_{w_R r_{\{s, t\}}} \star_{U_{w_R sts}} \left( V_{w_R str_{\{r, s\}}} \hat{\star} U_{w_R sr_{\{r, t\}}} \hat{\star} V_{w_R srr_{\{s, t\}}} \right) \\
		&\cong K_{R, s} \star_{U_{w_R srt}} V_{w_R srr_{\{s, t\}}}
	\end{align*}
	For the second claim we first want to show that the canonical map $O_R \star_{U_{w_R sr}} U_{w_R srs} \to K_{R, s} \star_{U_{w_R srt}} V_{w_R sr r_{\{s, t\}}}$	is injective. By Lemma \ref{CCleftCright} we know that $O_R \to K_{R, s} $ is injective. Moreover, $U_{w_R srs} \leq V_{w_R srr_{\{s, t\}}}$.
	
	Note that Proposition \ref{treeofgroupsinjective} implies $O_R \cap U_{w_R sr_{\{r, t\}}} = V_{w_R sr_{\{r, t\}}}$ in $K_{R, s}$ and hence 
	\[ O_R \cap U_{w_R srt} = O_R \cap U_{w_R srt} \cap V_{w_R sr_{\{r, t\}}} = O_R \cap U_{w_R sr} = U_{w_R sr}. \]
	Considering the preimage of the boundary monomorphisms the following hold:
	\allowdisplaybreaks
	\begin{align*}
		O_R \cap U_{w_R srt} = U_{w_R sr} = U_{w_R srs} \cap U_{w_R srt}
	\end{align*}
	As before, Proposition \ref{treeofgroupsinjective} implies that $O_R \star_{U_{w_R sr}} U_{w_R srs} \to K_{R, s} \star_{U_{w_R srt}} V_{w_R srr_{\{s, t\}}}$ is injective. Note that $O_{R, s} \cong O_R \star_{U_{w_R sr}} U_{w_R srs}$ by Lemma \ref{Lemma: V_R,s to O_R,s injective}.  Now Proposition \ref{treeofgroupsinjective} yields $O_{R, s} \cap K_{R, s} = O_R$. This finishes the claim.
\end{proof}

\subsection*{Tree products as subgroups of certain groups}

For $w_1, w_2 \in W$ we define $w_1 \prec w_2$ if $\ell(w_1) + \ell(w_1^{-1} w_2) = \ell(w_2)$. For any $w\in W$ we put $C(w) := \{ w' \in W \mid w' \prec w \}$.

\begin{definition}\label{Definition: Gi}
	We let
	\allowdisplaybreaks
	\begin{align*}
		&C_0 := \bigcup\nolimits_{S = \{r, s, t\}} \left( C(r_{\{s, t\}}) \cup C(rr_{\{s, t\}}) \right), \\
		&D_0 := \{ w_R r_{\{s, t\}} \mid R \text{ is of type } \{s, t\}, w_Rs, w_Rt \in C_0 \}
	\end{align*}
	and define $G_0$ to be the direct limit of the inductive system formed by the groups $(U_w)_{w\in C_0}$ and $(V_{w'})_{w' \in D_0}$, together with the natural inclusions $U_w \to U_{ws}$ if $\ell(ws) = \ell(w) +1$ and $U_{w_Rs} \to V_{w_Rr_{\{s, t\}}}$.
\end{definition}

\begin{remark}\label{Remark: G_i is generated by x_alpha}
	Note that $G_0$ is generated by elements $x_{\alpha, w}$ and $y_{\alpha, w'}$ for $w\in C_0$ and $w' \in D_0$, where $x_{\alpha, w}$ is a generator of $U_w$ and $y_{\alpha, w'}$ is a generator of $V_{w'}$. We first note that for every $w' = w_R r_{\{s, t\}}$ and every $\alpha \in \Phi_+$ with $w_R s \notin \alpha$, we have $x_{\alpha, w_R s} = y_{\alpha, w'}$ in $G_0$. Thus $G_0 = \langle x_{\alpha, w} \mid \alpha \in \Phi_+, w\in C_0, w \notin \alpha \rangle$. Moreover, we note that $G_0 = \langle x_{\alpha, r_{\{s, t\}}}, x_{\beta, rr_{\{s, t\}}} \mid \alpha, \beta \in \Phi_+, r_{\{s, t\}} \notin \alpha, rr_{\{s, t\}} \notin \beta \rangle$.
	
	There are twelve generators of the form $x_{\alpha, r_{\{s, t\}}}$ and $15$ generators of the form $x_{\beta, rr_{\{s, t\}}}$. Note, however, that each generator $x_{\alpha, r_{\{s, t\}}}$ coincides with some generator of the form $x_{\beta, rr_{\{s, t\}}}$. Thus $G_0 = \langle x_{\beta, rr_{\{s, t\}}} \mid \beta \in \Phi_+, rr_{\{s, t\}} \notin \beta \rangle$. Note that $\{ \alpha \in \Phi_+ \mid C_0 \not\subseteq \alpha \}$ contains also $15$ elements. This implies $G _0 = \langle x_{\alpha} \mid  \alpha \in \Phi_+, C_0 \not\subseteq \alpha \rangle$.
\end{remark}

\begin{lemma}\label{Lemma: G1 to mathcalG homomorphism}
	For all $w\in C_0$ and $w' \in D_0$ the canonical homomorphisms $U_w, V_{w'} \to G_0$ are injective. Moreover, for all $s\neq t \in S$ the canonical homomorphism $V_{R_{\{s, t\}}(1_W), s} \to G_0$ is injective.
\end{lemma}
\begin{proof}
	We abbreviate $R := R_{\{s, t\}}(1_W)$. Before we prove the claim we show that we have a canonical homomorphism $V_{R, s} \to G_0$. By Remark \ref{Remark: G_i is generated by x_alpha} it suffices to show that $srs, tr \in C_0$. But this holds by definition.
	
	Now we prove the claim. Let $\mathcal{D} = (\mathcal{G}, (U_{\alpha})_{\alpha \in \Phi})$ be the RGD-system associated with the split Kac--Moody group of type $(4, 4, 4)$ over $\FF_2$ as in Example \ref{exampleKM444}. We first show that we have canonical homomorphisms $U_w \to \mathcal{G}$ for each $w\in C_0$. Suppose $\alpha \in \Phi_+$ with $w\notin \alpha$. We show that the canonical mappings $x_{\alpha} \mapsto x_{\alpha} \in U_{\alpha} \leq \mathcal{G}$ extend to homomorphisms $U_w \to \mathcal{G}$. Let $\{\alpha, \beta\}$ be a pair of prenilpotent positive roots, let $w\in C_0$ and let $G \in \mathrm{Min}(w)$ be such that $\alpha \leq_G \beta \in \Phi(G)$. Suppose $o(r_{\alpha} r_{\beta}) < \infty$. As $\mathcal{M}$ is locally Weyl-invariant, we have 
	\[ M_{\alpha, \beta}^G = \begin{cases}
		(\alpha, \beta) & \text{ if }\vert (\alpha, \beta) \vert = 2, \\
		\emptyset & \text{ else}.
	\end{cases} \]
	We have seen in Example \ref{exampleKM444} that $[x_{\alpha}, x_{\beta}] = \prod_{\gamma \in M_{\alpha, \beta}^G} u_{\gamma}$ is also a relation in $\mathcal{G}$. Suppose now $o(r_{\alpha} r_{\beta}) = \infty$ and hence $\alpha \subsetneq \beta$. As $w\in C_0$, we deduce $(\alpha, \beta) = \emptyset$ and hence $[x_{\alpha}, x_{\beta}] = \prod_{\gamma \in M_{\alpha, \beta}^G} u_{\gamma} = 1$ does also hold in $\mathcal{G}$ by Example \ref{exampleKM444}. This implies that the mappings $x_{\alpha} \mapsto x_{\alpha}$ extend to a homomorphism $U_w \to \mathcal{G}$. To show that the mappings $x_{\alpha} \mapsto x_{\alpha}$ do also extend to a homomorphism $V_{w_R r_{\{u, v\}}} \to \mathcal{G}$, we have to show that the subgroup in $\mathcal{G}$ generated by $x_{w_R \alpha_u}, x_{w_R \alpha_v}$ has at most $8$ elements. As this is true, $x_{\alpha} \mapsto x_{\alpha}$ extend to a homomorphism $V_{w_R r_{\{u, v\}}} \to \mathcal{G}$. By definition the following diagrams commute:
	\begin{center}
		\begin{tikzcd}
			U_w \arrow[r] \arrow[rd] & U_{wu} \arrow[d] & \\
			&\mathcal{G}
		\end{tikzcd}
		\begin{tikzcd}
			U_{w_R u} \arrow[r] \arrow[rd] & V_{w_R r_{\{u, v\}}} \arrow[d] \\
			&\mathcal{G}
		\end{tikzcd}
	\end{center}
	The universal property of direct limits yields a unique homomorphism $G_0 \to \mathcal{G}$ extending $U_w, V_{w'} \to \mathcal{G}$. Note that $V_{R, s} \to \mathcal{G}$ is an injective homomorphism by Theorem \ref{theoremsubgroupKM}. Note that the following diagram commutes:
	\begin{center}
		\begin{tikzcd}
			V_{R, s} \arrow[r] \arrow[rd] & G_0 \arrow[d] \\
			& \mathcal{G}
		\end{tikzcd}
	\end{center}
	As $V_{R, s} \to \mathcal{G}$ is injective, $V_{R, s} \to G_0$ is injective as well and we are done.
\end{proof}

\begin{definition}
	\begin{enumerate}[label=(\alph*)]
		\item We let
		\allowdisplaybreaks
		\begin{align*}
			C_{-1} = \bigcup\nolimits_{s\neq t\in S} C(r_{\{s, t\}}) &&\text{and} && D_{-1} := \{ w_R r_{\{s, t\}} \mid R \text{ is of type } \{s, t\}, w_R s, w_R t \in C_{-1} \}
		\end{align*}
		and define $G_{-1}$ to be the direct limit of the groups $(U_w)_{w\in C_{-1}}$ and $(V_{w'})_{w' \in D_{-1}}$ as in Definition \ref{Definition: Gi}.
		
		\item For $S = \{r, s, t\}$ we let
		\allowdisplaybreaks
		\begin{align*}
			C_r := C(r_{\{r, s\}}) \cup C(r_{\{r, t\}}) &&\text{and} && D_r := \{ w_R r_{\{u, v\}} \mid R \text{ is of type } \{u, v\}, w_R u, w_R v \in C_r \}
		\end{align*}
		and define $G_{\{s, t\}}$ to be the direct limit of the groups $(U_w)_{w\in C_r}$ and $(V_{w'})_{w' \in D_r}$ as in Definition \ref{Definition: Gi}.
	\end{enumerate}
\end{definition}

\begin{remark}\label{Remark: generating set of G_0}
	\begin{enumerate}[label=(\alph*)]
		\item We note that there are nine roots $\alpha \in \Phi_+$ with the property $C_{-1} \not\subseteq \alpha$. Moreover, $G_{-1}$ is generated by $x_{\alpha, \{s,t\}}$ where $\alpha \in \Phi_+$ and $r_{\{s, t\}} \notin \alpha$. Thus $G_{-1}$ is generated by twelve elements. As $x_{\alpha_s, \{r, s\}} = x_{\alpha_s, \{s, t\}}$ in $G_{-1}$ for $S = \{r, s, t\}$, we deduce that $G_{-1}$ is generated by nine elements. In particular, the generator $x_{\alpha, w}$ does not depend on $w$. A similar result holds for $G_{\{s, t\}}$, which is generated by seven elements.
		
		\item For later references it will be useful to state the precise subsets $D_r, D_{-1}$ and $D_0$ here:
		\allowdisplaybreaks
		\begin{align*}
			&D_r = \{ r_{\{r, s\}}, r_{\{s, t\}}, r_{\{r, t\}}, rr_{\{s, t\}} \}, \\
			&D_{-1} = \{ r_{\{r, s\}}, r_{\{s, t\}}, r_{\{r, t\}}, rr_{\{s, t\}}, sr_{\{r, t\}}, tr_{\{r, s\}} \}, \\
			&D_0 = D_{-1} \cup \{ rs r_{\{r, t\}}, rt r_{\{r, s\}}, sr r_{\{s, t\}}, st r_{\{r, s\}}, tr r_{\{s, t\}}, ts r_{\{r, t\}} \}.
		\end{align*}
	\end{enumerate}	
\end{remark}

\begin{lemma}\label{OtoG-1}
	Let $s \neq t \in S$ and let $R := R_{\{s, t\}}(1_W)$. Then $V_{R, s} \to G_{\{s, t\}}$ is injective and $G_{-1} \cong G_{\{s, t\}} \star_{V_{R, s}} O_{R, s}$.
\end{lemma}
\begin{proof}
	As before, the assignments $x_{\alpha} \mapsto x_{\alpha}$ extend to homomorphisms $\pi: G_{\{s, t\}} \to G_0$ and $G_{\{s, t\}} \to G_{-1}$. Note that $srs, tr \in C_r \subseteq C_0$ and hence we have canonical homomorphisms $\phi: V_{R, s} \to G_{\{s, t\}}$ and $\psi: V_{R, s} \to G_0$. As $\psi = \pi \circ \phi$, Lemma \ref{Lemma: G1 to mathcalG homomorphism} implies that $\phi$ is injective. We abbreviate $H := G_{\{s, t\}} \star_{V_{R, s}} O_{R, s}$ (cf.\ Lemma \ref{Lemma: V_R,s to O_R,s injective}). Note that for each $w\in C(srs) \cup C(tr)$ the following diagram commuts:
	\begin{center}
		\begin{tikzcd}
			U_w \arrow[r] \arrow[d] & O_{R, s} \arrow[d] \\
			G_{\{s, t\}} \arrow[r] & G_{-1}
		\end{tikzcd}
	\end{center}
	The universal property of direct limits implies that there exists a unique homomorphism $H \to G_{-1}$. Now we want to construct a homomorphism $G_{-1} \to H$. Suppose that $S = \{r, s, t\}$. At first we recall that $G_{\{s, t\}}$ is generated by the seven elements $\{ x_{\alpha, G} \mid \alpha \in \Phi_+, C_r \not\subseteq \alpha \}$ and $O_{R, s}$ is generated by the seven elements $\{ x_{\alpha, O} \mid \alpha \in \Phi_+, \{srs, r_{\{s, t\}}, tr \} \not\subseteq \alpha \}$. In $H$ we have $x_{\alpha, G} = x_{\alpha, O}$ for $\alpha \in \{ \alpha_s, \alpha_t, s\alpha_r, sr\alpha_s, t\alpha_r \}$. Thus $H$ is generated by nine elements and we have a bijection between the set of generators of $H$ and the set of roots contained in $\{ \alpha \in \Phi_+ \mid C_{-1} \not\subseteq \alpha \}$. For $w\in C_r, w' \in D_r$ we have canonical homomorphisms $U_w, V_{w'} \to G_{\{s, t\}} \to H$. For $w\in C_{-1} \backslash C_r, w' \in D_{-1} \backslash D_r$ we have canonical homomorphisms $U_w, V_{w'} \to O_{R, s} \to H$. The universal property of direct limits yields a unique homomorphism $G_{-1} \to H$ extending the identities $U_w \to U_w \leq H$ and $V_{w'} \to V_{w'} \leq H$.
	
	Note that the concatenations of $H \to G_{-1}$ and $G_{-1} \to H$ fix all generators and hence they must be the identities. In particular, $H \to G_{-1}$ is an isomorphism.
\end{proof}

\begin{lemma}\label{OtoG0}
	For $R := R_{\{s, t\}}(r)$ the canonical homomorphisms $V_R, V_{R, s} \to G_{-1}$ are injective. For $D_R := G_{-1} \star_{V_R} O_R$ we obtain $D_R \cong G_{-1} \star_{V_{R, s}} O_{R, s}$. Moreover, we have $G_0 \cong \star_{G_{-1}} D_T$, where $T$ runs over $\mathcal{R}_1$.
\end{lemma}
\begin{proof}
	To show that we have canonical homomorphisms $V_R, V_{R, s} \to G_{-1}$ it suffices to check that $rsrs, rtr \in C_{-1}$. But this holds by definition.
	
	Let $T := R_{\{r, s\}}(1_W)$. By definition we have $O_{T, r} = V_{sr_{\{r, t\}}} \hat{\star} U_{r_{\{r, s\}}} \hat{\star} V_{rr_{\{s, t\}}} \hat{\star} U_{rtr}$ and $V_{R, s} = U_{r_{\{r, s\}}} \hat{\star} V_{rr_{\{s, t\}}} \hat{\star} U_{rtr}$. Using Proposition \ref{treeproducts} we obtain that $V_{R, s} \to O_{T, r}$ is injective. Using Lemma \ref{Lemma: V_R,s to O_R,s injective} and Lemma \ref{OtoG-1}, we obtain that each of the canonical homomorphisms $V_R \to V_{R, s} \to O_{T, r} \to G_{-1}$ is injective. Using Proposition \ref{treeproducts}, Remark \ref{Remark: isomorphism preserves amalgamated product}, Lemma \ref{folding}, Lemma \ref{VRtoORinjective} and Lemma \ref{Lemma: V_R,s to O_R,s injective} we obtain the following isomorphisms:
	\allowdisplaybreaks
	\begin{align*}
		G_{-1} \star_{V_R} O_R &\cong G_{-1} \star_{V_{R, s}} V_{R, s} \star_{V_R} O_R \cong G_{-1} \star_{V_{R, s}} \left( V_{R, s} \star_{V_R} O_R \right) \cong G_{-1} \star_{V_{R, s}} O_{R, s}
	\end{align*}
	
	It remains to show that $G_0 \cong \star_{G_{-1}} D_T$ holds, where $T$ runs over $\mathcal{R}_1$. Let $R \in \mathcal{R}_1$ be of type $\{s, t\}$. To see that we have a canonical homomorphism $O_R \to G_0$, it suffices to show that $rsr, rr_{\{s, t\}}, rtr \in C_0$. But this holds by definition. Using Remark \ref{Remark: G_i is generated by x_alpha} and Remark \ref{Remark: generating set of G_0}, we obtain a canonical homomorphism $\star_{G_{-1}} D_T \to G_0$, where $T$ runs over $\mathcal{R}_1$. Note that $\star_{G_{-1}} D_T$ is generated by the elements $x_{\alpha}, x_{\beta, T}$, where $C_{-1} \not\subseteq \alpha \in \Phi_+$ and $T \in \mathcal{R}_1$ is such that $T = R_{\{s, t\}}(r)$ and $\{ rsr, rr_{\{s, t\}}, rtr \} \not\subseteq \beta \in \Phi_+$. Note that if $\{ rsr, rtr \} \not\subseteq \beta \in \Phi_+$, then $x_{\beta} = x_{\beta, T}$ holds in $\star_{G_{-1}} D_T$. Thus $\star_{G_{-1}} D_T$ is generated by the elements $x_{\alpha}, x_{\beta, T}$, where $C_{-1} \not\subseteq \alpha \in \Phi_+$ and $C_{-1} \subseteq \beta$. Using Lemma \ref{mingallinrep} one can show that if $x_{\beta, T}$ and $x_{\beta', T'}$ are two such generators, then $T\neq T'$ implies $\beta \neq \beta'$. Thus the group $\star_{G_{-1}} D_T$, where $T$ runs over $\mathcal{R}_1$, is generated by $15$ elements and there is a bijection between the generators of $\star_{G_{-1}} D_T$ and the set $\{ \alpha \in \Phi_+ \mid C_0 \not\subseteq \alpha \}$. Hence the mappings $x_{\alpha} \mapsto x_{\alpha}$ extend to homomorphisms $U_w, V_{w'} \to \star_{G_{-1}} D_T$ for $w\in C_0$ and $w' \in D_0$. Note that the following diagrams commute:
	\begin{center}
		\begin{tikzcd}
			U_w \arrow[r] \arrow[rd] & U_{ws} \arrow[d] & \\
			&\star_{G_{-1}} D_T
		\end{tikzcd}
		\begin{tikzcd}
			U_{w_R s} \arrow[r] \arrow[rd] & V_{w_R r_{\{s, t\}}} \arrow[d] \\
			& \star_{G_{-1}} D_T
		\end{tikzcd}
	\end{center}
	The universal property of direct limits yields a unique homomorphism $G_0 \to \star_{G_{-1}} D_T$ extending $U_w, V_{w'} \to \star_{G_{-1}} D_T$. As the concatenations of $\star_{G_{-1}} D_T \to G_0$ and $G_0 \to \star_{G_{-1}} D_T$ fix $x_{\alpha}$, both concatenations are the identities and hence both homomorphisms are isomorphisms inverse to each other.
\end{proof}

\begin{lemma}\label{Lemma: K_Rs cap G_-1}
	Let $R \in \mathcal{T}_{0, 1}$ be a residue of type $J$. For $s\in J$ the canonical homomorphism $K_{R, s} \to D_{R_{\{r, t\}}(s)}$ is injective and we have $K_{R, s} \cap G_{-1} = O_R$ in $D_{R_{\{r, t\}}(s)}$.
\end{lemma}
\begin{proof}
	We suppose $J = \{s, t\}$. Note that $R = R_J(1_W)$. Since $O_{R, s} \cong U_{srs} \star_{U_{sr}} O_R$ by Lemma \ref{Lemma: V_R,s to O_R,s injective}, we obtain that both homomorphisms $O_R \to O_{R, s} \to G_{-1}$ are injective by Lemma \ref{OtoG-1} and Proposition \ref{treeproducts}. For $T := R_{\{r, t\}}(s)$ we have that $X := U_{sr_{\{r, t\}}} \hat{\star} V_{str_{\{r, s\}}} \to O_T$ is injective by Proposition \ref{treeproducts}. Using Corollary \ref{intersectionwithasubtree} we obtain $X \cap V_T = V_{sr_{\{r, t\}}} \hat{\star} U_{sts}$ in $O_T$. Note that $V_T \to O_{R, s}$ is injective by Lemma \ref{CleftCrightisos}. Recall that
	\allowdisplaybreaks
	\begin{align*}
		&O_{R, s} = U_{srs} \hat{\star} V_{sr_{\{r, t\}}} \hat{\star} U_{r_{\{s, t\}}} \hat{\star} V_{tr_{\{r, s\}}} &&\text{and} &V_T = U_{srs} \hat{\star} V_{sr_{\{r, t\}}} \hat{\star} U_{sts}.
	\end{align*}
	As $O_R$ corresponds to the last three vertex groups and $V_T$ is a subgroup of the first three vertex groups of $O_{R, s}$, Corollary \ref{intersectionwithasubtree} implies that $O_R \cap V_T = V_{sr_{\{r, t\}}} \hat{\star} U_{sts}$ in $O_{R, s}$. We define $Y := V_{sr_{\{r, t\}}} \hat{\star} U_{sts}$. Applying Proposition \ref{treeofgroupsinjective}, the canonical homomorphism $X \star_Y O_R \to O_T \star_{V_T} G_{-1} = D_T$ is injective. In particular, Proposition \ref{treeproducts}, Remark \ref{Remark: isomorphism preserves amalgamated product} and Lemma \ref{folding} yield
	\allowdisplaybreaks
	\begin{align*}
		X \star_Y O_R \cong X \star_Y \left( Y \star_{U_{sts}} U_{r_{\{s, t\}}} \hat{\star} V_{tr_{\{r, s\}}} \right) \cong U_{sr_{\{r, t\}}} \hat{\star} V_{str_{\{r, s\}}} \hat{\star} U_{r_{\{s, t\}}} \hat{\star} V_{tr_{\{r, s\}}} = K_{R, s}.
	\end{align*}
	This implies that $K_{R, s} \cong X \star_Y O_R \to O_T \star_{V_T} G_{-1} = D_T$ is injective. Applying Proposition \ref{treeofgroupsinjective}, we obtain $K_{R, s} \cap G_{-1} = O_R$ in $D_T$. This finishes the claim.
\end{proof}

\begin{theorem}\label{Theorem: Main application}
	For every $R \in \mathcal{T}_{0, 1}$ we have a canonical homomorphism $H_R \to G_0$. Moreover, this homomorphism is injective.
\end{theorem}
\begin{proof}
	Let $R \in \mathcal{T}_{0, 1}$ be of type $\{s, t\}$, i.e.\ $R = R_{\{s, t\}}(1_W)$. We first show that we have a homomorphism as claimed. By Remark \ref{Remark: G_i is generated by x_alpha} it suffices to show that $C_0$ contains the elements $sr_{\{r, t\}}, r_{\{s, t\}}, tr_{\{r, s\}}$. But this holds by definition.
	
	Next we show that this homomorphism is injective. Lemma \ref{CCleftCright} implies that $H_R \cong K_{R, s} \star_{O_R} K_{R, t}$. Thus it suffices to show that $K_{R, s} \star_{O_R} K_{R, t} \to G_0$ is injective. By Proposition \ref{treeproducts} and Lemma \ref{OtoG0} it follows that $D_{R_{\{r, t\}}(s)} \star_{G_{-1}} D_{R_{\{r, s\}}(t)} \to G_0$ is injective. Using Lemma \ref{Lemma: K_Rs cap G_-1} we obtain that $K_{R, s} \to D_{R_{\{r, t\}}(s)}$ and $K_{R, t} \to D_{R_{\{r, s\}}(t)}$ are injective and that $K_{R, s} \cap G_{-1} = O_R$ (resp.\ $K_{R, t} \cap G_{-1} = O_R$) in $D_{R_{\{r, t\}}(s)}$ (resp.\ $D_{R_{\{r, s\}}(t)}$). Now Proposition \ref{treeofgroupsinjective} implies that the following homomorphism is injective:
	\[ K_{R, s} \star_{O_R} K_{R, t} \to D_{R_{\{r, t\}}(s)} \star_{G_{-1}} D_{R_{\{r, s\}}(t)} \to G_0. \qedhere \]
\end{proof}

\begin{remark}
	Define $T = R_{\{r, t\}}(rs)$ and $T' = R_{\{r, s\}}(rt)$. In the following corollary we consider the group $\left( G_0 \star_{V_T} O_T \right) \star_{G_0} \left( G_0 \star_{V_{T'}} O_{T'} \right)$. Then $G_0 \star_{V_T} O_T$ is generated by $\{ G_0, x_{rsr\alpha_t, T}, x_{rst\alpha_r, T} \}$ and $G_0 \star_{V_{T'}} O_{T'}$ is generated by $\{ G_0, x_{rtr\alpha_s, T'}, x_{rts\alpha_r, T'} \}$. We first note that $C_0 \subseteq \beta$ for all $\beta \in \{ x_{rsr\alpha_t, T}, x_{rst\alpha_r, T}, x_{rtr\alpha_s, T'}, x_{rts\alpha_r, T'} \}$. As in the case of $H_{R_{\{s, t\}}(r)}$ one cas show that $\vert \{ rsr\alpha_t, rst \alpha_r, rtr\alpha_s, rts\alpha_r \} \vert = 4$. Thus the group $\left( G_0 \star_{V_T} O_T \right) \star_{G_0} \left( G_0 \star_{V_{T'}} O_{T'} \right)$ is generated by suitable $u_{\alpha}$.
\end{remark}

\begin{corollary}\label{Corollary: Main application}
	Define $R = R_{\{s, t\}}(r)$, $T = R_{\{r, t\}}(rs)$ and $T' = R_{\{r, s\}}(rt)$. Then $V_T, V_{T'} \to G_0$ are injective. Moreover, the canonical mapping $H_R \to \left( G_0 \star_{V_T} O_T \right) \star_{G_0} \left( G_0 \star_{V_{T'}} O_{T'} \right)$ is injective.
\end{corollary}
\begin{proof}
	Lemma \ref{CCleftCright} implies that $H_R \cong K_{R, s} \star_{O_R} K_{R, t}$. Thus it suffices to show that $K_{R, s} \star_{O_R} K_{R, t} \to \left( G_0 \star_{V_T} O_T \right) \star_{G_0} \left( G_0 \star_{V_{T'}} O_{T'} \right)$ is injective. We abbreviate $Z = R_{\{r, s\}}(1_W)$. By Theorem \ref{Theorem: Main application} the mapping $H_Z \to G_0$ is injective and, hence $V_T \to G_0$ is injective by Lemma \ref{Lemma: J_R,t cong H_R star V_T O_T}. By Lemma \ref{CleftCrightisos} the mappings $V_T \to O_{R, s}$ and, in particular, $K_{R, s} \to O_{R, s} \star_{V_T} O_T$ are injective. Lemma \ref{OtoG0} implies that the canonical homomorphisms $V_T \to O_{R, s} \to G_0$ are injective. Using Proposition \ref{treeofgroupsinjective} the homomorphism $O_{R, s} \star_{V_T} O_T \to G_0 \star_{V_T} O_T$ is injective. Using Proposition \ref{treeofgroupsinjective} again, we deduce $\left( O_{R, s} \star_{V_T} O_T \right) \cap G_0 = O_{R, s}$ in $G_0 \star_{V_T} O_T$ and hence $K_{R, s} \cap G_0 \leq O_{R, s}$ in $G_0 \star_{V_T} O_T$. By Lemma \ref{CleftCrightisos} we have $K_{R, s} \cap O_{R, s} = O_R$ in $O_{R, s} \star_{V_T} O_T$ and, hence, in $G_0 \star_{V_T} O_T$. Thus we obtain the following in $G_0 \star_{V_T} O_T$:
	\allowdisplaybreaks
	\begin{align*}
		K_{R, s} \cap G_0 = K_{R, s} \cap G_0 \cap O_{R, s} = K_{R, s} \cap O_{R, s} = O_R
	\end{align*}
	Replacing $s$ and $t$, we deduce that the homomorphisms $K_{R, t} \to O_{R, t} \star_{V_{T'}} O_{T'} \to G_0 \star_{V_{T'}} O_{T'}$ are injective and $K_{R, t} \cap G_0 = K_{R, t} \cap O_{R, t} = O_R$. Now Proposition \ref{treeofgroupsinjective} yields that $K_{R, s} \star_{O_R} K_{R, t} \to \left( G_0 \star_{V_T} O_T \right) \star_{G_0} \left( G_0 \star_{V_{T'}} O_{T'} \right)$ is injective. 
\end{proof}

\bibliographystyle{amsalpha}
\bibliography{references}

\end{document}